\documentclass[11pt,letterpaper]{amsart}

\usepackage{setspace}
\let\oldtocsection=\tocsection
\let\oldtocsubsection=\tocsubsection
\renewcommand{\tocsection}[2]{\hspace{-1mm}\bf\oldtocsection{ #1}{#2}}
\renewcommand{\tocsubsection}[2]{\hspace{6.9mm}\oldtocsubsection{#1}{#2}}

\usepackage[normalem]{ulem}
\usepackage{color}
\usepackage{amsmath}
\usepackage{amssymb}
\usepackage{array}
\usepackage{amsthm}
\usepackage[hidelinks]{hyperref}
\usepackage{xcolor}
\usepackage{enumitem}

\theoremstyle{plain}
\newtheorem{thm}{Theorem}[section]
\newtheorem{prop}[thm]{Proposition}
\newtheorem{cor}[thm]{Corollary}
\newtheorem{lem}[thm]{Lemma}

\theoremstyle{definition}
\newtheorem{defn}[thm]{Definition}

\newtheorem{exmp}[thm]{Example}

\theoremstyle{remark}
\newtheorem{rem}[thm]{Remark}

\newcommand{\m}{\mathfrak{m}}
\newcommand{\p}{\partial}

\newcommand{\R}{\mathbb{R}}

\newcommand{\N}{\mathbb{N}}

\newcommand{\grad}{\mathrm{grad}}

\newcommand{\identity}{\mathrm{Id}}
\newcommand{\scal}{\mathrm{scal}}
\newcommand{\ric}{\mathrm{Ric}}
\newcommand{\trace}{\mathrm{tr}}

\newcommand{\volume}{\mathrm{vol}}

\newcommand{\dv}{\, dV}

\def\hg{\widehat{g}}

\def\olg{\overline{g}}

\newcounter{mnotecount}[section]

\numberwithin{equation}{section}

\newtheorem{introthm}{\bf Theorem}

\title[A volume-renormalized mass for AH manifolds]{A volume-renormalized mass for asymptotically hyperbolic manifolds} 

\author[M. Dahl]{Mattias Dahl}
\address{Institutionen f\"or Matematik \\
Kungliga Tekniska H\"ogskolan \\
100 44 Stockholm \\
Sweden} 
\email{dahl@kth.se}

\author[K. Kr\"{o}ncke]{Klaus Kr\"{o}ncke}
\address{Institutionen f\"or Matematik \\
Kungliga Tekniska H\"ogskolan \\
100 44 Stockholm \\
Sweden} 
\email{kroncke@kth.se}

\author[S. McCormick]{Stephen McCormick}
\address{Institutionen f\"or teknikvetenskap och matematik \\
Lule{\aa} tekniska universitet \\
971 87 Lule\aa \\
Sweden} 
\email{stephen.mccormick@ltu.se}

\begin{document}

\begin{abstract}
We define a geometric quantity for asymptotically hyperbolic manifolds, which we call the volume-renormalized mass. It is essentially a linear combination of
the ADM mass surface integral and a renormalization of the volume.

We show that the volume-renormalized mass is well-defined and diffeomorphism invariant under weaker fall-off conditions than required to ensure that the renormalized volume and the ADM mass surface integral are well-defined separately. We prove several positivity results for the volume-renormalized mass. We also use it to define a renormalized Einstein--Hilbert action and a renormalized expander entropy which is nondecreasing under the Ricci flow. Further, we show that local maximizers of the entropy are local minimizers of the volume-renormalized mass.
\end{abstract}

\maketitle
\vspace{-8mm}

\onehalfspacing
\tableofcontents
\singlespacing

\section{Introduction}

For manifolds asymptotic to hyperbolic space, there is an established concept of asymptotically hyperbolic mass defined through a Hamiltonian formulation of Einstein's equations in general relativity, see \cite{ChruscielHerzlich2003, Wang2001mass}. The reference spacetime is the static Anti-de Sitter metric.

In this paper, we introduce a new mass-like quantity for the larger class of asymptotically hyperbolic (AH) manifolds that are asymptotically Poincar\'e--Einstein (APE), which means that $\ric+(n-1)g$ decays at an appropriate rate. By AH we mean that the manifolds are conformally compact with sectional curvature tending to $-1$ towards the conformal boundary. This new mass can also be deduced from a reduced Hamiltonian formulation of Einstein's equation, see \cite{DKM_forthcoming,FischerMoncrief2002}. In this case, the reference spacetime is an expanding Milne-type metric.




Given two AH manifolds $(M^n,g)$ and $(\widehat{M}^n,\hg)$ with diffeomorphic conformal infinities,
the \textit{volume-renormalized mass} of $g$ with respect to $\hg$ is defined as
\begin{equation}\begin{split}\label{eq-defn0}
\mathfrak{m}_{\rm{VR},\hg}(g)&=\int_{\partial_{\infty} M} (\mathrm{div}_{\hg}(\varphi_*g)-d\trace_{\hg}(\varphi_*g))(\nu)dA\\
&\qquad +2(n-1)\left(\int_M \dv_g-\int_{\widehat{M}}\dv_{\hg} \right),
\end{split}
\end{equation}
where $\varphi$ is a diffeomorphism between neighborhoods of the conformal infinities such that $\varphi_*g-\hg$ decays suitably, and the linear combination of integrals should be understood as an appropriate limit. 



When the asymptotic fall-off is so fast that the boundary integral vanishes, the volume-renormalized mass is simply the renormalized volume. In this sense, we can also view the quantity as a generalization of the renormalized volume. 
 Positivity of the renormalized volume has been proven for metrics on $\mathbb R^3$ asymptotic to the standard hyperbolic metric by Brendle and Chodosh \cite{BrendleChodosh2014}, which can thus be viewed as a positive mass theorem for the volume-renormalized mass under strong decay conditions.

Let us mention that the definition of APE manifolds given above is  convenient for the following reason: If $(M^n,g)$ and $(\widehat{M}^n,\hg)$ are APE manifolds with isometric conformal boundaries, there exists a diffeomorphism $\varphi$ between neighborhoods of the conformal infinities such that $\varphi_*g-\hat{g}=O(e^{-\delta r})$ for some $\delta>\frac{n-1}{2}$.
This is a consequence of the well-known Fefferman--Graham expansion for AH Einstein metrics near the conformal boundary, see Proposition \ref{prop:APE}.

We are now ready to state the main results of this paper.
\begin{introthm}\label{mainthm:mass_well_def}
Let $(M^n,g)$ and $(\widehat{M}^n,\hg)$ be APE manifolds with isometric conformal boundaries that both satisfy $\scal+n(n-1)\in L^1$.
Then $\mathfrak{m}_{\rm{VR},\hg}(g)$ is well-defined and finite. 
\end{introthm}
For the proof of the theorem, we study a renormalized version of the Einstein--Hilbert action for AH manifolds.

%
%

A priori, the definition of the volume-renormalized mass depends on the choice of  $\varphi$. From a physical perspective however, a mass should be a coordinate-invariant object and therefore not depend on the choice of diffeomorphism $\varphi$. We are indeed able to show that this is the case for the volume-renormalized mass, provided that an additional condition holds.
\begin{introthm}\label{mainthm:mass_coord_inv}
Let $(M^n,g)$ and $(\widehat{M}^n,\hg)$ be APE manifolds with isometric conformal boundaries which both satisfy $\scal+n(n-1)\in L^1$.
  If the conformal boundaries are proper, $\mathfrak{m}_{\rm{VR},\hg}(g)$ does not depend on the choice of $\varphi$.
\end{introthm}
In this context, we call a conformal class proper if it is the conformal boundary of a PE manifold $(\overline{M},\overline{g})$ such that every isometry of the conformal boundary extends to an isometry of $(\overline{M},\overline{g})$.
It is easy to see that the conformal class of the round sphere is proper.
It is known that every conformal class of a smooth metric on a compact manifold is the conformal class of a PE manifold, see \cite{GurSze20}. It seems reasonable to believe that every such conformal class is  proper, but we do not have a proof of this conjecture at the moment.

We also show that the mass satisfies an additivity property, see Proposition \ref{cor:additivity_renormalized_EH}. As a consequence, the functional $g\mapsto \mathfrak{m}_{\rm{VR},\hg}(g)$ only changes by a constant if we change the reference metric $\hg$. Summarizing, we have for every proper conformal boundary a natural mass functional which is diffeomorphism invariant and well-defined up to a constant.

We prove the following positive mass theorem for two-dimensional manifolds.
\begin{introthm}\label{mainthm:PMT_2D}
Consider a surface $(M^2,g)$ asymptotic to $\mathbb{R}^2$ with the metric $\widehat{g}  = dr^2+\sinh^2(r)\left(\frac{\omega}{2\pi}\right)^2d\theta^2$, 
which is the hyperbolic metric with angular defect $\omega$.
Under the assumption that $\scal_g+2$ is nonnegative and integrable we have  
\begin{align}	
\m_{\rm{VR},\widehat{g}}(g)+2(2\pi-\omega) \geq 0,
\end{align}
where equality holds if and only if $(M^2,g)$ is isometric to $(\widehat{M},\widehat{g})$.
\end{introthm}
For three-dimensional manifolds we prove the following.
\begin{introthm}\label{mainthm:PMT_3D}
Let $g$ be a complete APE metric on $\R^3$ whose conformal boundary is the hyperbolic metric $g_{\rm hyp}$. Assume furthermore that $\scal_g+6$ is nonnegative and integrable. Then $\mathfrak{m}_{{\rm VR},g_{\rm  hyp}}(g)$ is nonnegative and vanishes if and only if $g$ is isometric to $g_{\rm hyp}$.
\end{introthm}
The proof of this theorem uses  positivity of the renormalized volume by Brendle and Chodosh \cite{BrendleChodosh2014}, combined with a density argument and the following  conformal positive mass theorem.
\begin{introthm}\label{mainthm:conformal_positive_mass}
Let $(M^n,\hg)$ be a complete APE manifold with $\scal_{\hg}=-n(n-1)$, and let $(M^n,g)$ be a complete APE manifold conformal to $(M^n,\hg)$. 
Then if $\scal_g + n(n-1)$ is nonnegative and integrable, we have $\mathfrak{m}_{\rm{VR},\hg}(g) \geq 0$ with equality only if $g = \hg$.
\end{introthm}

Furthermore, we use the volume-renormalized mass to define a renormalized expander entropy
 $g\mapsto \mu_{AH,\hg}(g)$ 
 for AH manifolds,
  which is monotone under the (normalized) Ricci flow 
$\partial_tg=-2\ric_g-2(n-1)g$
and whose critical points are PE. This part of the article is inspired by work of Deruelle and Ozuch \cite{deruelle2020ojasiewicz} who use the ADM mass to define a version of Perelman's $\lambda$-functional for asymptotically locally Euclidean (ALE) manifolds which is monotone under the Ricci flow in its standard form on ALE manifolds. However, their functional is a priori only defined near a Ricci-flat manifold and seems not to be defined for every ALE metric.
In contrast, our version of the expander entropy is defined for every APE manifold.

Our final main result is a local positive mass theorem which is as follows.
\begin{introthm}\label{mainthm:local_PMT}
Let $(M,\hg)$ be a complete PE manifold. Then the following two assertions are equivalent:
\begin{itemize}
\item[(i)] $\hg$ is a local maximiser of $\mu_{\rm{AH},\hg}$
\item[(ii)] $\hg$ is a local minimum of $\mathfrak{m}_{\rm{VR},\hg}$ among all metrics with $\scal+n(n-1)$ being nonnegative and integrable.
\end{itemize}
Furthermore, we have:
\begin{itemize}
\item[(a)] If $(M,\hg)$ is linearly stable and integrable, then (i) and (ii) hold.
\item[(b)] If (i) and (ii) hold, then  $(M,\hg)$ is scalar curvature rigid under a volume constraint.
\end{itemize}
\end{introthm}
Ilmanen conjectured a relation between Ricci flow and the positive mass theorem, partly proven in \cite[Proposition 0.1]{deruelle2020ojasiewicz} and \cite[Theorem 8.1]{HHS_Ricciflat_cones_2014} for ALE manifolds. Theorem \ref{mainthm:local_PMT} solves an AH version of this conjecture.
 In the article \cite{kroencke2023dynamical}, Yudowitz and the second author prove that a PE manifold is stable under the Ricci flow if and only if (i) and (ii) in Theorem \ref{mainthm:local_PMT} are satisfied.

%

The structure of the paper is as follows. In Section \ref{Sec-Notation}, we explain some notation and make some of the definitions of this introduction more precise.
 In Section \ref{sec:well-defined}, we prove Theorem \ref{mainthm:mass_well_def} and \ref{mainthm:mass_coord_inv} and we study the renormalized Einstein--Hilbert action.
In Section \ref{Sec:PMTs} we prove Theorem \ref{mainthm:PMT_2D}, Theorem \ref{mainthm:conformal_positive_mass} and Theorem \ref{mainthm:PMT_3D}.
In Section \ref{sec:entropy}, we define and study the renormalized expander entropy $\mu_{\rm{AH},\hg}$. 
Then finally in Section \ref{sec:localPMT}, we establish Theorem \ref{mainthm:local_PMT} as a combination of Theorem \ref{thm:local_maximum_entropy}, Theorem \ref{thm_local_positive_mass} and Corollary \ref{cor:scr} in Section \ref{sec:localPMT}.

\subsection*{Acknowledgements}
We are grateful to Eric Bahuaud, Piotr Chru\'sciel, Marc Herzlich, Jan Metzger and Eric Woolgar for helpful comments and inspiring discussions.

\section{Notation and definitions} \label{Sec-Notation}

Throughout the article, we use $n$ to denote the dimension of the manifold
We assume $n\geq 3$ unless otherwise specified. For the Laplacian we use the sign convention $\Delta=-\mathrm{div}\circ d=-\trace\circ\nabla^2$.

\begin{defn}
Let $N$ be a compact manifold with compact boundary $\partial N$. Let $\rho : N \to [0,\infty)$ be a smooth {\em boundary defining function}, which means that $\rho^{-1}(0) = \partial N$ and $d\rho|_{\partial N} \neq 0$. Let $M = N \setminus \partial N$. We say that a Riemannian metric $g$ on $M$ is {\em conformally compact} of class $C^{k,\alpha}$, if there is a $C^{k,\alpha}$-Riemannian metric $b$ on $N$ 
so that $g = \rho^{-2} b$. In this case,
the sectional curvatures of $g$ tend to $-|d\rho|^2_{h}$ at $\partial N$. If $|d\rho|^2_{h} = 1$ so that all sectional curvatures tend to $-1$ at $\partial N$ we say that $(M,g)$ is {\em asymptotically hyperbolic}, or simply AH. The Riemannian manifold $(N,b)$ is called the {\em conformal background} of $(M,g)$. With $\sigma=b|_{\partial N}$, we call $(\p N,[\sigma])$ the {\em conformal boundary} of $(M,g)$.
\end{defn}
\begin{rem}
We also call a manifold $(M,g)$ conformally compact if it is the complement of a compact set of a manifold that is conformally compact in the above sense.
\end{rem}
%

Throughout, we will make use of a radial function $r$ defined by $\rho =e^{-r}$. We will work in weighted H\"older spaces 
 $C^{k,\alpha}_\delta(M)=e^{-\delta r}C^{k,\alpha}(M)$, equipped with the norm
\begin{align*}
\|u\|_{k,\alpha,\delta} = \|e^{\delta r}u\|_{C^{k,\alpha}(M)}.
\end{align*}
Here, $\delta\in\mathbb R$ and $C^{k,\alpha}(M)$ denotes the standard H\"older space with the norm $\|\cdot\|_{C^{k,\alpha}(M)}$.
Weighted H\"older spaces of sections of bundles are defined similarly,
see Lee \cite{Lee06} for further details.

For a fixed AH manifold $(\widehat{M}, \hg)$, 
we define the space of Riemannian metrics on $\widehat{M}$ asymptotic to $\hg$ as
\begin{align*}
\mathcal{R}^{k,\alpha}_{\delta}(\widehat{M},\hg)
=\left\{ g \mid g-\hg\in C^{k,\alpha}_{\delta}(S^2_+T^*\widehat{M})\right\},
\end{align*}
where $S^2_+T^*\widehat{M}$ is the bundle of positive definite symmetric bilinear forms on $\widehat{M}$. 

\begin{defn}\label{defn-AH}
Let $(M,g)$, $(\widehat{M} ,\hg)$ be AH manifolds with conformal backgrounds $N,\widehat{N}$, respectively. We say that $(M,g)$ is \emph{asymptotic to} $(\widehat{M} ,\hg)$ of order $\delta>0$ with respect to $\varphi$, if there are
bounded and closed sets $K\subset M$, $\widehat{K}\subset \widehat{M}$ and a $C^{k+1,\alpha}$-diffeomorphism $\varphi: N \setminus K\to \widehat{N}\setminus \widehat{K}$
of manifolds with boundary such that 
$\varphi_*g \in \mathcal{R}^{k,\alpha}_{\delta}(\widehat{M} \setminus \widehat{K}, \hg)$.
\end{defn}
With $\hg, K,\widehat{K}$ and $\varphi$ as in Definition \ref{defn-AH}, we define
\begin{align*}
\mathcal{R}^{k,\alpha}_{\delta}(M,\varphi,\hg)=\left\{ g \mid g\in C^{k,\alpha}(S^2_+T^*M), \varphi_*g \in \mathcal{R}^{k,\alpha}_{\delta}(\widehat{M} \setminus \widehat{K}, \hg)\right\}.
\end{align*}

Given an AH manifold $(M,g)$ asymptotic to $(\widehat{M} ,\hg)$ with diffeomorphism $\varphi$, we choose the boundary defining functions on $N$ and $\widehat{N}$ so that $\rho = \widehat{\rho} \circ \varphi$
and $r = \widehat{r} \circ \varphi$ on $N\setminus K$. Define the sets
\begin{align*}
B_R=\{x\in {M} \mid {r}(x)<R \} \subset M, \qquad
\p B_R=\{x\in {M} \mid {r}(x)=R \} \subset{M},
\end{align*}
and
\begin{align*}
\widehat{B}_R=\{x\in \widehat{M} \mid \widehat{r}(x)<R\}\subset\widehat{M},\qquad
\p\widehat{B}_R=\{x\in \widehat{M} \mid \widehat{r}(x)=R\}\subset\widehat{M}.
\end{align*}
For $R$ so large that $\varphi(\p B_R)=\p\widehat{B}_R$ for $R$, let
\begin{align*}
\m^\varphi_{{\rm ADM},\widehat{g}}(g,R)
&= \int_{\p\widehat{B}_R}(\mathrm{div}_{\hg}(\varphi_*g)-d\trace_{\hg}(\varphi_*g))(\nu_{\hg})\dv_{\widehat{g}},\\
RV^\varphi_{\widehat{g}}(g,R)
&= \int_{B_R}\dv_g-\int_{\widehat{B}_{R}}\dv_{\widehat{g}},
\end{align*}
where $\nu_{\hg}$ is the outward unit normal to $\p \widehat B_R$ in $(\widehat M,\hg)$.

\begin{defn}\label{defn-VRmass}
Let $(M,\hg)$ be asymptotically hyperbolic. 
We define the {\em volume-renormalized mass} $\mathfrak{m}^\varphi_{\rm{VR},\hg}(g)$ of $g$ with respect to $\hg$ and $\varphi$ as
\begin{equation} \label{eq-first-m-defn}
\mathfrak{m}^\varphi_{\rm{VR},\hg}(g)=
\lim_{R \to \infty}
\left( \m^\varphi_{{\rm ADM},\widehat{g}}(g,R)+2(n-1)RV^\varphi_{\widehat{g}}(g,R) \right).
\end{equation}
\end{defn}
Theorem \ref{thm-well-def-mass} in the following section demonstrates that this quantity is well-defined under the assumptions that
$\varphi_*g \in \mathcal{R}^{k,\alpha}_{\delta}(\widehat{M} \setminus \widehat{K}, \hg)$ for some $\delta>\frac{n-1}{2}$, where $\hg$ is APE (in the sense of Definition \ref{defn:APE_2} below)
and the scalar curvature satisfies $\scal_{\hg}+n(n-1)\in L^1(M)$. 

Given a boundary defining function $\rho$, an AH metric $g$ can be written as
\begin{align}\label{eq:expansion_PE}
g=\rho^{-2}(d\rho^2+\widehat{\sigma}_{\rho}),
\end{align}
where $\widehat{\sigma}_{\rho}$ is a family of metrics on $\partial N$.
If $g$ is PE, it is known from the work of Fefferman and Graham \cite{FefGra85}, that $\widehat{\sigma}_{\rho}$ has the asymptotic expansion
\begin{align}\label{eq:expansion_PE_even}
\widehat{\sigma}_{\rho}=\sigma_0+\rho^2\sigma_2+\rho^3\sigma_3+\ldots +\rho^{n-2}\sigma_{n-2}+\rho^{n-1}\sigma_{n-1}+O(\rho^n)
\end{align} 
if $n$ is even and
 \begin{align}\begin{split}\label{eq:expansion_PE_odd}
 \widehat{\sigma}_{\rho}&=\sigma_0+\rho^2\sigma_2+\rho^3\sigma_3+\ldots +\rho^{n-3}\sigma_{n-3}\\&\qquad+\rho^{n-1}(\sigma_{n-1}+\log(\rho)\tilde{\sigma}_{n-1})+O(\rho^n\log(\rho)),\end{split}
 \end{align}
if $n$ is odd.
 Here, $\sigma_0=b|_{\partial N}$ and the tensors $\sigma_i$, $2\leq i\leq n-2$, 
are uniquely determined by $\sigma_0$. In the odd case, the tensor $\tilde{\sigma}_{n-1}$ is also determined by $\sigma_0$. 
The metric $\sigma_0$ together with the first undetermined term $\sigma_{n-1}$,  determine all remaining terms of the asymptotic expansion.
 
\begin{defn}\label{defn:APE_2} 
We say an AH manifold $(M,g)$  of class $C^{k,\alpha}$, $k\geq 2$ is \emph{asymptotically Poincar\'{e}--Einstein} (APE) of order $\delta$ if
$|\ric_g+(n-1)g|_g\in C^{k-2,\alpha}_\delta(M)$ for some  $\frac{n-1}{2}<\delta<n-1$ satisfying $\delta \leq k+\alpha$. The set of all such metrics is denoted by
$\mathcal{R}^{k,\alpha}_{\delta}(M)$.
\end{defn}
\begin{prop}\label{prop:APE}
Let $(M,g)$ and $(\widehat{M},\widehat{g})$ be $C^{k,\alpha}$-asymptotically hyperbolic manifolds with isometric conformal boundaries. 
Then, if both manifolds are APE of order $\delta$, $(M,g)$ is asymptotic to $(\widehat{M},\widehat{g})$ of order $\delta$.
\end{prop}
\begin{proof}
This follows from the Fefferman--Graham expansion discussed above.
Fix a boundary defining function $\rho$ and write $g$ as in \eqref{eq:expansion_PE}
The family $\widehat{\sigma}_{\rho}$ is now a $C^{k,\alpha}$-family of metrics on $\partial N$ with $\widehat{\sigma}_{\rho}|_{\rho=0}=\sigma_0$. 
By lowering the regularity, we may assume that $\delta=k+\alpha$. Using the APE condition, the arguments of Bahuaud, Mazzeo and Woolgar \cite[Prop. 2.2]{BMW2015}, then show that
\begin{align*}
\widehat{\sigma}_{\rho}=\sigma_0+\rho^2\sigma_2+\ldots +\rho^{k}\sigma_k+O(\rho^{k+\alpha}),
\end{align*}
where the tensor fields $\sigma_2, \dots, \sigma_k$ are uniquely determined by $\sigma_0$.
Fix a boundary defining function $\widehat{\rho}$ on $\widehat{M}$ such that $\widehat{\sigma}_0=\widehat{\rho}^2\hg|_{\partial \widehat{N}}$ is isometric to $\sigma_0$. Let $\varphi:N\setminus K\to \widehat{N}\setminus \widehat{K}$ a diffeomorphism such that $\widehat{\rho}\circ \varphi=\rho$ and $(\varphi|_{\partial N})^*\widehat{\sigma}_0=\sigma_0$. Repeating the above arguments, we get
\begin{align*}
\varphi^*\hg=\rho^{-2}(d\rho^2+\widetilde{\sigma}_{\rho})
\end{align*}
with 
\begin{align*}
\widetilde{\sigma}_{\rho}=\sigma_0+\rho^2\sigma_2+\ldots +\rho^{k}\sigma_k+O(\rho^{k+\alpha}),
\end{align*}
so that $\varphi^*\hg-g\in O(\rho^{k+\alpha})$, or equivalently, $\varphi_*g-\hg\in O(\widehat{\rho}^{k+\alpha})$, as desired. 
\end{proof}

\begin{rem}
A similar, but stronger, definition of APE manifolds 
for a different purpose 
is given in \cite{BMW2015}, where they use the weight $\delta=n$.
\end{rem}


\section{Well-definedness and coordinate invariance of the mass}\label{sec:well-defined}

\subsection{A renormalized Einstein--Hilbert action}

We utilize an AH version of the Einstein--Hilbert action to establish well-definedness of the volume-renormalized mass.
\begin{thm}\label{thm-well-def-mass}
Let $(\widehat M,\hg)$ be an APE manifold with 
$\scal_{\hg}+n(n-1)\in L^1$. Then for $(M,g)$ asymptotic to $(\widehat M,\hg)$ of order $\delta>\frac{n-1}2$, the limit
\begin{align*}
S^{\varphi}_{\hg}(g) = \lim_{R \to \infty}
\Big( 
& \int_{B_R} \left( \scal_g + n(n-1) \right) \dv_g\\
&\qquad -\m^{\varphi}_{{\rm ADM},\widehat{g}}(g,R) 
- 2(n-1)RV^{\varphi}_{\widehat{g}}(g,R)
\Big).
\end{align*}
is well-defined and finite, where $\varphi$ is the diffeomorphism from Definition \ref{defn-AH}. In particular, $\m^{\varphi}_{{\rm VR},\widehat{g}}(g)$ is well-defined and finite if $\scal_g + n(n-1)\in L^1(M)$. 
\end{thm}
\begin{defn}
We call the functional $g\mapsto S^{\varphi}_{\hg}(g)$ the
\em{renormalized Einstein--Hilbert action}.
\end{defn}
\begin{proof}[Proof of Theorem \ref{thm-well-def-mass}]
Fix some large $R_0$ so that $K\subset B_{R_0}$. For $R>R_0$ we define the annular regions $A_R=B_R\setminus B_{R_0}$ and $\widehat A_R=\phi(A_R)$. This allows us to work with $\varphi_* g$ on $A_R$.
	
We now note that $\m^\varphi_{{\rm ADM},\widehat{g}}(g,R)$ can be expressed via the divergence theorem in terms of the linearization of scalar curvature as
\begin{align*}
\m^\varphi_{{\rm ADM},\widehat{g}}(g,R)
&=
\int_{\widehat B_R} \left(\mathrm{div}_{\hg}(\mathrm{div}_{\hg}(\varphi_*g))+\Delta_{\hg}(\trace_{\hg}(\varphi_*g))\right)\, \dv_{\hg}\\
&=
\int_{\widehat A_R} \left(D\scal_{\hg}[h]+\langle h,\ric_ {\hg}\rangle\right)\, \dv_{\hg}+C,
\end{align*}
where $h=\varphi_*g-\hg$ and $C$ is the finite contribution from the integral over $B_{R_0}$. We will use $C$ throughout the proof to denote such a finite term independent of $R$, where the exact value may vary from line to line.

By Taylor expanding the scalar curvature we may write the above expression as
\begin{align*}
\m^\varphi_{{\rm ADM},\widehat{g}}(g,R)
&= \int_{\widehat A_R} \left( \scal_{\varphi_*g}-\scal_{\hg} 
+ \langle h,\ric_{\hg}\rangle + Q_1(h)
\right) \dv_{\hg}+C\\
&= \int_{\widehat A_R} \left( \scal_{\varphi_*g}-\scal_{\hg} 
+ \langle h,\ric_{\hg}\rangle
\right) \dv_{\hg}+C.
\end{align*}
Here, $Q_1(h)$ is a remainder term quadratic in $h$ and its first two derivatives. For the second equality, we have used that $Q(h)$ is integrable, which follows from $h\in C^{2,\alpha}_{\delta}$ with $\delta>\frac{n-1}{2}$.
Similarly, we Taylor expand the volume form
\begin{align*}
\dv_{\phi_*g}
= \left(1+ \frac{1}{2} \trace_{\hg}(h) +Q_2(h)\right)\dv_{\hg},
\end{align*}
where $Q_2(h)$ is quadratic in $h$ and integrable.
We get
\begin{align*}
RV^\varphi_{\widehat{g}}(g,R) 
= \frac{1}{2} \int_{\widehat A_R}\trace_{\hg}(h) \dv_{\hg}
+ C.
\end{align*}
Now we bring everything together to obtain
\begin{align*}
S^\varphi_{\hg}(g)
&= \lim_{R \to \infty} \Big(
\int_{B_R} \left( \scal_{g} +n(n-1) \right)\dv_g- \m^\varphi_{{\rm ADM},\widehat{g}}(g,R)
\\&\qquad
-2(n-1)RV^\varphi_{\widehat{g}}(g,R)\Big) \\
&= \lim_{R \to \infty} \Big(\int_{\widehat A_R} \left( \scal_{\varphi_*g} +n(n-1) \right)\dv_{\varphi_*g}- \m^\varphi_{{\rm ADM},\widehat{g}}(g,R)
\\&\qquad
-2(n-1)RV^\varphi_{\widehat{g}}(g,R)\Big)+C\\
&= \lim_{R \to \infty} \Big(\int_{\widehat A_R} \left( \scal_{\varphi_*g} +n(n-1) \right)\dv_{\varphi_*g}\\
&\qquad -\int_{\widehat A_R} \left(\scal_{\varphi_*g}-\scal_{\hg} 
+ \langle h,\ric_{\hg}\rangle
+(n-1)\trace_{\hg}(h) \right)\dv_{\hg})\Big)+C\\
&= \lim_{R \to \infty} \Big(\int_{\widehat A_R} \left( \scal_{\varphi_*g} +n(n-1) \right)(\dv_{\varphi_*g}-\dv_{\hg})\\
&\qquad +\int_{\widehat A_R} \left(\scal_{\hg} +n(n-1)
- \langle h,\ric_{\hg}+(n-1)\hg\rangle
\right)\dv_{\hg}\Big)+C.
\end{align*}
By assumption, $\scal_{\hg}+n(n-1)\in L^1$. Since $g$ is APE, $\ric_{\hg}+(n-1)\hg$ and $\scal_{\varphi_*g} +n(n-1)$ are in $L^2$. Since also $h\in L^2$, we conclude that the limit is finite.
\end{proof}
From Definition \ref{defn:APE_2} and Proposition \ref{prop:APE}, we immediately obtain:
\begin{cor}\label{cor:mass_well_def}
Let $(M,g)$ and $(\widehat{M},\widehat{g})$ be APE manifolds with isometric conformal boundaries and assume that $\scal_{\hg}+n(n-1)\in L^1$. Then, $S^{\varphi}_{\hg}(g)$ is well-defined and finite, where $\varphi$ is a diffeomorphism in the sense of Definition \ref{defn-AH}. If furthermore $\scal_g+n(n-1)\in L^1$, then $\m^{\varphi}_{{\rm VR},\widehat{g}}(g)$ is well-defined and finite.
\end{cor}
We will often use Corollary \ref{cor:mass_well_def} without giving an explicit reference to it. Furthermore, for the sake of convenience, we will say that an APE manifold has \emph{integrable normalized scalar curvature}
 if $\scal+n(n-1)\in L^1$.

\begin{rem}
The definition is of the volume-renormalized mass uses an exhaustion of $M$ by coordinate balls.
 The proof of Theorem~\ref{thm-well-def-mass} demonstrates that the volume-renormalized mass is independent of the choice of exhaustion. In particular, it is independent of the pair of boundary defining functions on $M$ and $\widehat{M}$, as long as these are $\varphi$-compatible.
\end{rem}

\begin{rem}Note that under the conditions of Theorem \ref{thm-well-def-mass},
$\m^{\varphi}_{{\rm VR},\widehat{g}}(g)=\pm\infty$  if $\pm(\scal_g + n(n-1))\geq0$ and $\scal_g + n(n-1)\notin L^1(M)$. In particular, this is independent of the chosen diffeomorphism $\varphi$.
\end{rem}
We have the following additivity properties for the renormalized Einstein--Hilbert action and for the volume-renormalized mass. A similar additivity property for mass invariants of asymptotically hyperbolic manifolds is studied in  \cite{CG21}.
\begin{prop}\label{cor:additivity_renormalized_EH}
Let $(M,g),(\widetilde M,\widetilde g),(\widehat M,\hg)$ be APE manifolds
with isometric conformal boundaries and assume that $(\widetilde M,\widetilde g)$ and $(\widehat M,\hg)$ have integrable normalized scalar curvature. 
Then the renormalized Einstein--Hilbert action satisfies the additivity property
\begin{equation*}
S^{\widetilde\varphi\circ\varphi}_{\hg}(g)
= S^\varphi_{\widetilde g}(g)
-\mathfrak{m}^{\widetilde\varphi}_{\rm{VR},\hg}(\widetilde{g}),
\end{equation*} 
where $\varphi:N\setminus K\to \widetilde N\setminus \widetilde K$ and $\widetilde \varphi:\widetilde N\setminus \widetilde K\to \widehat{N}\setminus\widehat{K}$, respectively, are diffeomorphisms as in Definition~\ref{defn-AH}. If $(M,g)$ also has integrable normalized scalar curvature, it holds that
\begin{align*}
\mathfrak{m}^{\widetilde\varphi\circ\varphi}_{\hg}(g)
= \mathfrak{m}^{\varphi}_{\widetilde g}(g)
+\mathfrak{m}^{\widetilde\varphi}_{\rm{VR},\hg}(\widetilde{g}).
\end{align*}
\end{prop}
\begin{proof}
For $R$ sufficiently large, we have
\begin{align*}
RV^{\tilde{\varphi}\circ \varphi}_{\widehat{g}}(g,R)=
RV^{ \varphi}_{\widetilde g}(g,R)+
RV^{\tilde{\varphi}}_{\widehat{g}}(\widetilde g,R)
\end{align*}
so that
\begin{align*}
S^\varphi_{\widetilde g}(g)-S^{\widetilde\varphi\circ\varphi}_{\hg}(g)
&=\lim_{R\to\infty} \left(
\m^{\widetilde{\varphi}\circ\varphi}_{{\rm ADM},\widehat{g}}(g,R)-\m^{\varphi}_{{\rm ADM},\widetilde{g}}(g)
+2(n-1)RV^{\tilde{\varphi}}_{\widehat{g}}(\tilde{g},R)
\right)
\end{align*}
Let $\gamma=\widetilde{\varphi}_*\varphi_*g$. By diffeomorphism invariance, we get 
\begin{align*}
&\m^{\widetilde{\varphi}\circ\varphi}_{{\rm ADM},\widehat{g}}(g,R)-\m^{\varphi}_{{\rm ADM},\widetilde{g}}(g,R) \\
&\qquad=
\m^{\identity}_{{\rm ADM},\widehat{g}}(\widetilde{\varphi}_*\varphi_*g,R)-\m^{\identity}_{{\rm ADM},\widetilde{\varphi}_*\widetilde{g}}(\widetilde{\varphi}_*\varphi_*g,R)\\
&\qquad=
\m^{\identity}_{{\rm ADM},\widehat{g}}(\gamma,R)-\m^{\identity}_{{\rm ADM},\widetilde{\varphi}_*\widetilde{g}}(\gamma-\widetilde{\varphi}_*\widetilde{g},R)\\
&\qquad=
\m^{\identity}_{{\rm ADM},\widehat{g}}(\gamma,R)-\m^{\identity}_{{\rm ADM},\widehat{g}}(\gamma-\widetilde{\varphi}_*\widetilde{g},R)\\
&\qquad\qquad+
\m^{\identity}_{{\rm ADM},\widehat{g}}(\gamma-\widetilde{\varphi}_*\widetilde{g},R)-\m^{\identity}_{{\rm ADM},\widetilde{\varphi}_*\widetilde{g}}(\gamma-\widetilde{\varphi}_*\widetilde{g},R)\\
&\qquad=
\m^{\identity}_{{\rm ADM},\widehat{g}}(\widetilde{\varphi}_*\widetilde{g},R)\\
&\qquad\qquad+
\m^{\identity}_{{\rm ADM},\widehat{g}}(\gamma-\widetilde{\varphi}_*\widetilde{g},R)-\m^{\identity}_{{\rm ADM},\widetilde{\varphi}_*\widetilde{g}}(\gamma-\widetilde{\varphi}_*\widetilde{g},R)\\
&\qquad=\m^{\widetilde{\varphi}}_{{\rm ADM},\widehat{g}}(\widetilde{g},R)\\
&\qquad\qquad+
\m^{\identity}_{{\rm ADM},\widehat{g}}(\gamma-\widetilde{\varphi}_*\widetilde{g},R)-\m^{\identity}_{{\rm ADM},\widetilde{\varphi}_*\widetilde{g}}(\gamma-\widetilde{\varphi}_*\widetilde{g},R).
\end{align*}
Let $h=\gamma-\widetilde{\varphi}_*\widetilde{g}$ and $\overline{g}=\widetilde{\varphi}_*\widetilde{g}$.
For the last two terms, we have
\begin{align*}
\left|\m^{\identity}_{{\rm ADM},\widehat{g}}(h,R)-\m^{\identity}_{{\rm ADM},\overline{g}}(h,R)\right|\leq C\int_{\partial \widehat{B}_R} 
\left(|\widehat{\nabla}(\widehat{g}-\overline{g})||h|+|\widehat{g}-\overline{g}||\widehat{\nabla} h|\right)\dv_{\widehat{g}}
\end{align*}
Because $\widehat{g}-\overline{g},h\in C^{k,\alpha}_{\delta}$ with $\delta>\frac{n-1}{2}$, this converges to $0$ as $R\to\infty$, which proves the proposition.
\end{proof}

We now compute the first variation of the renormalized Einstein--Hilbert action.
\begin{prop}\label{prop:first_variation_EHaction}
The first variation of $S^\varphi_{\hg}$ is given by
\begin{equation} \label{eq:firstvarS}
D_gS^\varphi_{\hg}[h]=
-\int_M \left\langle \ric_ g-\frac{1}{2}\scal_g\, g-\frac{1}{2}(n-2)(n-1)g,h\right\rangle_g \dv_g.
\end{equation}
In particular, $S_{\hg}^\varphi$ is an analytic functional on the space $\mathcal{R}^{k,\alpha}_{\delta}(M,\varphi,\hg)$ for any $\delta>\frac{n-1}2$.
\end{prop}

\begin{proof}
Let $g_t=g+th$. We compute
\begin{align*}
&\frac{d}{dt}\int_{B_R}\left( \scal_{g_t}+n(n-1)\right)\dv_{g_t}\big|_{t=0}\\
&\qquad=\int_{B_R}\left( D_g\scal[h]+\left( \scal_{g}+n(n-1)\right)\frac{1}{2}\trace_gh\right)\dv_{g}\\
&\qquad=\int_{B_R}\left(\mathrm{div}_g(\mathrm{div}_g(h))+\Delta_g(\trace_g h)-\langle h, \ric_ g \rangle\right)\dv_g\\
&\qquad\qquad +\int_{B_R}\left( \scal_{g}+n(n-1)\right)\frac{1}{2}\trace_gh\dv_{g} \\
&\qquad=\m^{\identity}_{{\rm ADM},g}(h,R)-\int_{B_R}\langle\ric-\frac{1}{2}(\scal_g+n(n-1))g,h\rangle_g \dv_g,
\end{align*}
where we use the divergence theorem in the last line. We further compute
\begin{align*}
\frac{d}{dt}RV^{ \varphi}_{\widehat{g}}(g_t,R)|_{t=0}=\frac{1}{2}\int_{B_R}\trace_gh\dv_g=\frac{1}{2}\int_{B_R}\langle g,h\rangle \dv_g
\end{align*}
and
\begin{align*}
\frac{d}{dt}\m^{\varphi}_{{\rm ADM},\hg}(g_t,R)|_{t=0}=
\m^{\varphi}_{{\rm ADM},\hg}(h,R).
\end{align*}
As in the proof of Proposition \ref{cor:additivity_renormalized_EH}, we find that
\begin{align*}
\lim_{R\to\infty}\left( \m^{\identity}_{{\rm ADM},g}(h,R)- \m^{\varphi}_{{\rm ADM},\hg}(h,R)\right)=0.
\end{align*}
Adding up the above identities and letting $R\to\infty$ thus finishes the proof.
\end{proof}
The proposition demonstrates that $S^{\varphi}_{\hg}$ is a natural version of the Einstein--Hilbert action for asymptotically hyperbolic manifolds, illustrated also by the following corollary.
\begin{cor}\label{cor-EH}{\,}
\begin{itemize}
\item[(i)] If $n=2$, the functional $g\mapsto S^\varphi_{\hg}(g)$ is constant on $\mathcal{R}^{k,\alpha}_{\delta}(M,\varphi,\hg)$.
\item[(ii)] If $n\geq 3$, the critical points of $S^\varphi_{\hg}$ are exactly the PE metrics.
\end{itemize}
\end{cor}

\subsection{Diffeomorphism invariance}
In this subsection, we demonstrate that the volume-renormalized mass is a diffeomorphism-invariant quantity.
A first indication is given by the following lemma, which requires slightly more regularity than we have previously assumed. 

\begin{lem} \label{lem:smalldiffeoinvar}
Let $(M,g)$ and $(\widehat{M},\hg)$ be APE manifolds with isometric conformal boundaries and assume that $(\widehat{M},\hg)$ has integrable normalized scalar curvature. Assume also that the manifolds are $C^{k,\alpha}$, with $k\geq3$.
Then for all $X\in C^{k,\alpha}_\delta(TM)$, we have 
\begin{align*}D_gS^\varphi_{\hg}(\mathcal{L}_Xg)=0.
\end{align*}
\end{lem}

\begin{proof}
Since $\mathcal L_X g\in C^{k-1,\alpha}_\delta$, Proposition \ref{prop:first_variation_entropy} gives us
\begin{equation*}
D_gS_{\hg}^\varphi(\mathcal L_X g)
=
-\int_M \left\langle \ric_ g-\frac{1}{2}\scal_g\cdot g-\frac{1}{2}(n-1)(n-2)g,
\mathcal L_X g\right\rangle_g \dv_g.
\end{equation*}
Integrating by parts we get a boundary term at infinity which vanishes due to the decay conditions. This leaves us with
\begin{equation*}
D_gS_{\hg}^\varphi(\mathcal L_X g)
=
-2\int_M \mathrm{div}\left( \ric_ g-\frac{1}{2}\scal_g\cdot g\right)(X) \dv_g,
\end{equation*}
which vanishes by the contracted second Bianchi identity.
\end{proof}

An immediate consequence is the following result.

\begin{cor}\label{cor:diff_inv:EH}
Let $(M,g)$, $(\widehat{M},\hg)$ and $\varphi$ be as in Lemma \ref{lem:smalldiffeoinvar}. Additionally let $X\in C^{k,\alpha}_{\delta}(TM)$ and $\psi_t$ be the group of diffeomorphisms generated by $X$. Then,
\begin{align*}
S^\varphi_{\hat{g}}((\psi_t)_*g)=S^\varphi_{\hg}(g)
\end{align*}
for all $t\in\R$. If in addition $\scal_g+n(n-1)\in L^1$, we obtain
\begin{align*}
\m^\varphi_{{\rm VR},\hg}((\psi_t)_*g)=\m^\varphi_{{\rm VR},\hg}(g)
\end{align*}
for all $t\in\R$.
\end{cor}

This corollary asserts that the mass does not change if we modify the $C^{k,\alpha}$-diffeomorphism $\varphi: N \setminus K\to \widehat{N}\setminus \widehat{K}$ slightly by a diffeomorphism on $N$, generated by a vector field with sufficiently fast falloff. This is does not give a completely satisfying answer to the question of diffeomorphism invariance. For example, it excludes Lorentz boosts on hyperbolic space, which are generated by vector fields that do not decay towards the conformal boundary.

From now on, we may again allow $C^{k,\alpha}$, with $k\geq2$ and $k+\alpha\geq \delta$.
The following lemma is entirely straightforward, however we state it for completeness.
\begin{lem}\label{lem:bothslotdiffeo}
Let $(M,g)$ and $(\widehat{M},\hg)$ be APE manifolds with integrable normalized scalar curvature and isometric conformal boundaries
and let $\psi:M\to M$ be a diffeomorphism. Then
\begin{equation*}
\m_{{\rm VR},\hg}^\varphi(g)=\m_{\hg}^{\varphi\circ\psi^{-1}}(\psi_*g),
\end{equation*}
where $\varphi$ is the usual defining diffeomorphism as in Definition \ref{defn-AH}. Similarly, if $\chi:\widehat{M}\to \widehat{M}$ is a diffeomorphism,
\begin{equation*}
\m_{{\rm VR},\hg}^\varphi(g)=\m_{\chi_ *\hg}^{\chi\circ \varphi}(g).
\end{equation*}
\end{lem}
\begin{proof}
The renormalized volume clearly is unchanged by a diffeomorphism and the ADM term is unchanged because $(\varphi\circ\psi^{-1})_*(\psi_* g)=\varphi_* g$. The other argument is analogous.
\end{proof}
In the following, we are going to prove a much stronger type of diffeomorphism invariance under a natural condition. 
\begin{defn}
Let $(\widehat{M},\widehat{g})$ be a (possibly incomplete) PE manifold with conformal boundary $(\partial \widehat{N},[\widehat{\sigma}])$. We call $(\widehat{M},\widehat{g})$ {\em proper} if every conformal isometry $\psi_{\partial}\in \mathrm{Iso}(\partial \widehat{N},[\widehat{\sigma}])$ extends to a $C^{k+1,\alpha}$-diffeomorphism ${\psi}:\widehat{N}\to \widehat{N}$
which restricts to an isometry of $(\widehat{M},\widehat{g})$. We call a conformal class $(\partial \widehat{N},[\widehat{\sigma}])$ proper if it is the conformal boundary of a proper PE manifold.
\end{defn}
The key technical step
is provided by the following lemma.
\begin{lem}\label{lem:coord_invariance_mass}
Let $(\widehat{M},\widehat{g})$ be a proper PE manifold. Then,
for every pair of open neighborhoods ${U},{V}$ of $\partial \widehat{N}$ in $\widehat{N}$ and
 every $C^{k+1,\alpha}$-diffeomorphism $\varphi:U\to V$ with $\varphi_*\widehat{g}\in \mathcal{R}^{k,\alpha}_{\delta}(V,\hg)$, we have that

\begin{align*}
\m^{\varphi}_{{\rm VR},\hg}(\hg)=0.
\end{align*}
\end{lem}

\begin{proof}
Because $\varphi$ restricts to a $C^{k+1,\alpha}$-diffeomorphism on the boundary and $\varphi_*\widehat{g}-\widehat{g}\in C^{k,\alpha}_{\delta}$, we get that $\varphi|_{\partial \widehat{N}}$ is a conformal isometry on the boundary. Since $(\widehat{M},\widehat{g})$ is proper, 
there exists a $C^{k+1,\alpha}$ diffeomorphism $\psi$ on $\widehat{N}$ with $\psi|_{\partial \widehat{N}}=\varphi|_{\partial \widehat{N}}$ which restricts to an isometry of $(\widehat{M},\hg)$. The diffeomorphism $\chi=\varphi^{-1}\circ\psi:\psi(V)\to U$ satisfies 
$\chi_*\widehat{g}-\widehat{g}\in  C^{k,\alpha}_{\delta}$ and 
$\chi|_{\partial \widehat{N}}=\identity_{\partial\widehat{N}}$. Furthermore, 
by Lemma \ref{lem:bothslotdiffeo} and since $\psi$ is an isometry we have
\begin{align*}
\m^{\varphi}_{{\rm VR},\hg}(\hg)
=\m^{\varphi\circ\psi^{-1}}_{{\rm VR},\hg}(\psi_*\hg)=\m^{\chi^{-1}}_{{\rm VR},\hg}(\hg).
\end{align*}
Let us extend $\psi$ to a $C^{k+1,\alpha}$-diffeomorphism on $\widehat{M}$, again denoted by $\psi$. This does not change the value of the mass and by Lemma \ref{lem:bothslotdiffeo}, we have
\begin{align*}
\m^{\chi^{-1}}_{{\rm VR},\hg}(\hg)=\m^{\identity}_{{\rm VR},\hg}(\chi^*\hg).
\end{align*}
It therefore remains to show that $\m^{\identity}_{{\rm VR},\hg}(\chi^*\hg)=0$. 

Changing the diffeomorphism  $\chi$ inside a ball $B_{R}$ does not change the mass. By deforming $\chi$ on a bounded subset, we may therefore assume that $\chi=\identity$ inside such a ball. Choosing $R$ sufficiently large will make the difference $\chi-\identity$ arbitrarily small (measured with respect to the conformal background $\widehat{h}$).
Provided that $R$ is chosen large enough, we therefore find a $C^{k+1,\alpha}$-vector field $X$ on $\widehat{N}$ such that the flow $\chi_t$ generated by $X$ satisfies $\chi_1=\chi$. Note that $X$ vanishes on $\partial\widehat{N}$ and $B_R$.
 The family $t\mapsto \hg_t=\chi_t^*\hg$ then connects $\hg$ and $\chi^*\hg$. 
  By Taylor expansion in time,
\begin{align}\label{eq:metric_curve_taylor_expansion}
\chi^*\widehat{g}-\widehat{g}
=\int_{0}^1\frac{d}{dt} \widehat{g}_tdt
=\int_{0}^1\frac{d}{dt} (\chi_{t}^*\widehat{g})dt
=\int_{0}^1\chi_{t}^*(\mathcal{L}_X\widehat{g})dt.
\end{align}
Our goal is now to determine the weighted regularity of $\frac{d}{dt} \widehat{g}_t$.

For this purpose, let $\rho$ be a boundary defining function for $\partial\widehat{N}$ and denote by $y$ the coordinates on $\partial\widehat{N}$. By Taylor expansion at $\rho=0$ and because $X$ vanishes on the boundary, we get that 
\begin{align}\label{eq:expansion_vectorfield}
X=a\rho \partial_\rho+ \rho Y(\rho)+O(\rho^2),
\end{align}
where $Y(\rho)$ is a $\rho$-dependant family of vector fields on $\partial\widehat{N}$ and $a$ is a function on the boundary. With respect to the conformal background metric $\widehat{b}$, we have 
\begin{align*}
|\nabla^{(l)}_{\hg}X|_{\hg}=\rho^{l-1}|\nabla^{(l)}_{\hg}X|_{\widehat{b}}.
\end{align*}
To compare covariant derivatives of $\hg$ and $\widehat{h}$, we note that
\begin{align*}
\Gamma(\hg)_{ij}^l-\Gamma(\widehat{b})_{ij}^l
=
-2\rho^{-1}(\delta_j^l\partial_i\rho +\delta_i^l\partial_j\rho-\widehat{b}^{lm}\widehat{b}_{ij}\partial_m\rho )
=
\nabla^{(1)}´_{\widehat{h}}\rho*\rho^{-1},
\end{align*}
where $*$ denotes a linear combination of tensor products and contractions using the metric $\widehat{h}$.
An induction argument then shows that
\begin{align*}
\nabla^{(l)}_{\hg}X-\nabla^{(l)}_{\widehat{b}}X
=
\sum_{l_0+\ldots +l_r+r=l} \left(\frac{\nabla^{(l_1+1)}_{\widehat{h}}\rho}{\rho}*\ldots *
\frac{\nabla^{(l_r+1)}_{\widehat{h}}\rho}{\rho}\right)*\nabla^{(l_0)}_{\widehat{h}}X.
\end{align*}
From this, we find
\begin{align*}
|\nabla^{(l)}_{\hg}X|_{\hg}=\rho^{l-1}|\nabla^{(l)}_{\hg}X|_{\widehat{b}}
\leq C\sum_{m=0}^l\rho^{m-1}|\nabla^{(m)}_{\widehat{b}}X|_{\widehat{b}},
\end{align*}
which implies
\begin{align*}
\left\| X\right\|_{C^{k+1,\alpha}(\hg)}\leq \left\| \nabla_{\widehat{b}}^{(1)} X\right\|_{C^{k,\alpha}(\widehat{b})}+\left\| \rho^{-1}X\right\|_{C^{0}(\widehat{b})}.
\end{align*}
Since $X$ is $C^{k+1,\alpha}$-regular and vanishes on the boundary, we have $|X|_{\widehat b}=O(\rho)$ so that $\left\| X\right\|_{C^{k+1,\alpha}(\hg)}<\infty$. Therefore, the family of metrics $t\mapsto \hg_t=\chi_t^*\hg$ is a smooth family in $\mathcal{R}^{k,\alpha}_{0}(\widehat{M}, \hg)$. In particular, the $C^{k,\alpha}$-norms of the metrics $\hg_t$, $t\in [0,1]$ are uniformly equivalent and $\frac{d}{dt}\hg_t=\chi_{t}^*(\mathcal{L}_X\widehat{g})\in C^{k,\alpha}(\chi_{t}^*\hg)=C^{k,\alpha}(\hg)$.
Next, we will improve the weight of this regularity.
By \eqref{eq:expansion_vectorfield}, there are constants $a_-<a_+$, depending on the bounds of the function $a$ such that
\begin{align*}
e^{a_-t}\rho(x)\leq \rho(\chi_t(x))\leq e^{a_+t}  \rho(x)
\end{align*}
for all $x\in P$ and $t\in\R$.
We therefore find uniform constants $C_1,C_2>0$ such that
\begin{align*}
C_1\rho \leq \chi_t^*\rho\leq C_2\rho
\end{align*}
for all $t\in [0,1]$.
Thus for any weight $\delta\in\R$ and $t\in[0,1]$, the weighted spaces $C^{k,\alpha}_{\delta}$ generated by the metrics $\hg_t$ are also uniformly equivalent and we have that $h\in C^{k,\alpha}_{\delta}(S^2T^*\widehat{M},\hg)$ (or equivalently, $\chi_t^*h\in  C^{k,\alpha}_{\delta}(S^2T^*\widehat{M},\chi_t^*\hg)$) if and only if $\chi_t^*h\in  C^{k,\alpha}_{\delta}(S^2T^*\widehat{M},\hg)$. 
Since $\chi^*\widehat{g}-\widehat{g}\in C^{k,\alpha}_{\delta}$, the expansion \eqref{eq:metric_curve_taylor_expansion} implies that $\frac{d}{dt}\hg_t\in C^{k,\alpha}_{\delta}(S^2T^*\widehat{M},\hg)$ for $t\in [0,1]$.

Thus, $\hg_t$ is a smooth family of isometric PE metrics in $C^{k,\alpha}_{\delta}(S^2_+T^*\widehat{M},\hg)$ connecting $\hg$ and $\chi^*\hg$.
The function $t\mapsto S^{\identity}_{\widehat{g}}(\chi_{t}^*\widehat{g})$ is constant, as it is evaluated along a family of critical metrics. In addition, 
$S^{\identity}_{\widehat{g}}(\chi_{t}^*\widehat{g})=-\m^{\identity}_{{\rm VR},\widehat{g}}(\chi_{t}^*\widehat{g})$ because all $\chi_{t}^*\widehat{g}$ have constant scalar curvature $-n(n-1)$.
Therefore, 
\begin{align*}
\m^{\identity}_{{\rm VR},\hg}(\chi^*\hg)=\m^{\identity}_{{\rm VR},\hg}(\hg)=0,
\end{align*}
which was to be proven.
\end{proof}

\begin{thm}	\label{thm:diffeoinvar}
Let $(M,g)$ and $(\widehat M,\hg)$ be APE manifolds with integrable normalized scalar curvature and isometric conformal boundaries. Assume that the conformal boundaries are proper. Then, the definition of $\m^{\varphi}_{\rm{VR},\hg}(g)$ is independent of the choice of $C^{k+1,\alpha}$-diffeomorphism $\varphi$.
\end{thm}	

\begin{proof}	

Let $\varphi_1:N\setminus K_1\to \widehat{N}\setminus \widehat{K}_1$ and $\varphi_2:N\setminus K_2\to \widehat{N}\setminus \widehat{K}_2$ be  $C^{k+1,\alpha}$-diffeomorphisms with the property that 	
\begin{align*}	
(\varphi_1)_*g\in \mathcal{R}^{k,\alpha}_{\delta}(\widehat{M}\setminus \widehat{K}_1,\hg),
\qquad		
(\varphi_2)_*g\in \mathcal{R}^{k,\alpha}_{\delta}(\widehat{M}\setminus \widehat{K}_2,\hg).
\end{align*}	
Our goal is to show that	
\begin{align*}	
\m^{\varphi_1}_{\rm{VR},\hg}(g)=\m^{\varphi_2}_{\rm{VR},\hg}(g).	
\end{align*}
Let $(\overline{M},\overline{g})$ be a proper PE manifold with the same conformal boundary as $(\widehat{M},\widehat{g})$.	Let furthermore 
$\chi:\widehat{N}\setminus \widehat{K}\to \overline{N}\setminus\overline{K}$ be a $C^{k+1,\alpha}$-diffeomorphism such that	
\begin{align*}	
\chi_*\hg\in 
\mathcal{R}^{k,\alpha}_{\delta}(\overline{M}\setminus \overline{K},\overline{g}).	
\end{align*}	
Assume without loss of generality that $\widehat{K}\subset \widehat{K}_1\cap \widehat{K}_2$ so that the diffeomorphisms $\psi_i=\chi\circ\varphi_i:N\setminus K_i\to \overline{N}\setminus \overline{K}_i$, $i=1,2$ are defined. 
We then get	
\begin{align}	
\label{eq_psi1}(\psi_1)_*g-\overline{g} &\in C^{k,\alpha}_{\delta}(S^2T^*(\overline{M}\setminus \overline{K}),\overline{g}),\\
\label{eq_psi2}
(\psi_2)_*g-\overline{g}&\in C^{k,\alpha}_{\delta}(S^2T^*(\overline{M}\setminus \overline{K}),\overline{g}),
\end{align}	
and Proposition \ref{cor:additivity_renormalized_EH}  implies that	
\begin{align*}	
\m^{\psi_i}_{\rm{VR},\overline{g}}(g)
=\m^{\varphi_i}_{\rm{VR},\hg}(g) + \m^{\chi}_{\rm{VR},\overline{g}}(\hg)
\end{align*}	
for $i=1,2$. Thus, it suffices to show 	
\begin{align*}	
\m^{\psi_1}_{\rm{VR},\overline{g}}(g)=\m^{\psi_2}_{\rm{VR},\overline{g}}(g).
\end{align*}
For this purpose, observe that we may choose closed and bounded subsets 
$\overline{K}_1$ and $\overline{K}_2$ of $\overline{N}$ such that we get a diffeomorphism
\begin{align*}
\psi_2\circ (\psi_1)^{-1}:\overline{N}\setminus \overline{K}_1
\to \overline{N}\setminus \overline{K}_2.
\end{align*}
Note that \eqref{eq_psi2} implies 
\begin{align*}
C^{k,\alpha}_{\delta}(S^2T^*(\overline{M}\setminus \overline{K}_2),(\psi_2)_*g)=C^{k,\alpha}_{\delta}(S^2T^*(\overline{M}\setminus \overline{K}_2).\overline{g})
\end{align*}
By \eqref{eq_psi1} and diffeomorphism invariance, we thus get
\begin{align*}
(\psi_2)_*g - (\psi_2\circ \psi_1^{-1})_*	 \overline{g}
&\in 
C^{k,\alpha}_{\delta}(S^2T^*(\overline{M}\setminus \overline{K}_2),(\psi_2)_*g)\\
&\qquad=C^{k,\alpha}_{\delta}(S^2T^*(\overline{M}\setminus \overline{K}_2),\overline{g}).
\end{align*}
Together with \eqref{eq_psi2}
and the triangle inequality, we obtain
\begin{align*}
(\psi_2\circ \psi_1^{-1})_* \overline{g}-\overline{g}
\in 
C^{k,\alpha}_{\delta}(S^2T^*(\overline{M}\setminus \overline{K}_2),\overline{g}).
\end{align*}
We can also extend $\psi_2\circ \psi_1^{-1}$ to a $C^{k+1,\alpha}$-diffeomorphism 
$\theta:\overline{M} \to \overline{M}$. Thus, $\psi_2=\theta\circ \psi_1$ and 
\begin{align*}
\theta_* \overline{g}-\overline{g}
\in 
C^{k,\alpha}_{\delta}(S^2_+T^*(\overline{M}\setminus \overline{K}_2),\overline{g}).
\end{align*}	
Proposition \ref{cor:additivity_renormalized_EH} together with Lemma \ref{lem:coord_invariance_mass} implies that
\begin{align*}
\m^{\psi_2}_{\rm{VR},\overline{g}}(g)
&=
\m^{\psi_2}_{\rm{VR},\theta* \overline{g}}(g)	
- \m^{\identity}_{\rm{VR},\theta_*\overline{g}}(	\overline{g})\\
&=
\m^{\psi_2}_{\rm{VR},\theta*	\overline{g}}(g)
+\m^{\identity}_{\rm{VR},\overline{g}}(\theta_* 	\overline{g})\\
&=
\m^{\theta\circ \psi_1}_{\rm{VR},\theta_*	\overline{g}}(g) \\
&=
\m^{\psi_1}_{\rm{VR},\overline{g}}(g)		,	
\end{align*}
where we used Lemma \ref{lem:bothslotdiffeo} for the last equality.
\end{proof}

\begin{rem}
Note that the key arguments are indeed performed in the proof of Lemma \ref{lem:coord_invariance_mass}. Interestingly, the proof does not require a choice of preferred coordinate system, which may be contrasted to the proofs of diffeomorphism invariance of the ADM mass, see \cite{Bartnik86,Chrusciel1986}. 
\end{rem}

\begin{rem}
The Fefferman--Graham expansions \eqref{eq:expansion_PE_even}, \eqref{eq:expansion_PE_odd} imply that two PE metrics $(M,g)$ and $(\widehat{M},\hg)$ with isometric conformal boundaries are asymptotic to each other of order $n-1$.
%
  Therefore the ADM boundary term is finite and computed from the first undetermined term $\sigma_{n-1}$ in the Fefferman--Graham expansion.
 Thus, the renormalized volume is finite as well. 
The functional $g\mapsto \m^{\varphi}_{\rm{VR},\hg}(g)$
is constant along a family of complete Einstein metrics by Corollary \ref{cor-EH}, but the renormalized volume and ADM boundary term may vary. 
\end{rem}

%
%

\section{Special positive mass theorems}\label{Sec:PMTs}
%

For the remainder of this paper, we will assume that all 
 AH manifolds have proper conformal boundaries, so that Theorem \ref{thm:diffeoinvar} holds.
We may thus drop the dependence of the mass and the renormalized Einstein--Hilbert action on the diffeomorphism in the notation and  write $S_{\hg}(g)$ instead of  $S_{\hg}^{\varphi}(g)$ and $\m_{\rm{VR},\hg}(g)$ instead of $\m_{\rm{VR},\hg}^{\varphi}(g)$.

\subsection{Two-dimensional positive mass theorem}

In Corollary \ref{cor-EH}, we have seen that the functional $g\mapsto S_{\hg}(g)$ is constant in dimension two. We will now determine the value of this constant and deduce a positive mass theorem for surfaces.

As a reference surface, we use the hyperbolic plane with an angle defect. That is, $(\widehat{M},\widehat{g})=(\mathbb{R}^2,g_{\mathrm{hyp},\omega})$, where the metric $g_{\mathrm{hyp},\omega}$ is given by
\begin{align*}	
g_{\mathrm{hyp},\omega}
= dr^2+\sinh^2(r)\left(\frac{\omega}{2\pi}\right)^2d\theta^2, 
\qquad \theta\in [0,2\pi] \text{ mod }2\pi.	
\end{align*}
in polar coordinates.

\begin{thm}\label{thm-2d}	
Consider a complete APE surface $(M^2,g)$ which is asymptotic to $(\widehat{M},\widehat{g})$. We then have	
\begin{align}\label{eq:S_dim_two}
S_{\widehat{g}}(g) = 4\pi(\chi(\overline{M})-2) + 2(2\pi-\omega),	
\end{align}	
where $\overline{M}$ is the one-point compactification of $M$. Under the assumption that $\scal_g+2\in L^1$ we conclude the identity
\begin{align}\label{eq:m_dim_two}	
\m_{\rm{VR},\widehat{g}}(g)+2(2\pi-\omega)
=
\int_M (\scal_g+2)\dv_g
+ 4\pi(2-\chi(\overline{M})).	
\end{align}
If in addition, $\scal_g\geq -2$, we find that 
\begin{align}\label{eq:pmt_dim_two}	
\m_{\rm{VR},\widehat{g}}(g)+2(2\pi-\omega) \geq 0,
\end{align}
where equality holds if and only if $(M^2,g)$ is isometric to $(\widehat{M},\widehat{g})$. In this case, $(\widehat{M},\widehat{g})$ is also complete, and thus $\omega=2\pi$.
\end{thm}

\begin{proof}	
By linear interpolation we deform $g$ to a metric $\overline{g}$ such that $\varphi_* \overline{g} = \widehat{g}$ on $\widehat{M}\setminus \widehat{B}_{R_0-\epsilon}$ for some large radius $R_0$. From  Corollary \ref{cor-EH} we have 	
\begin{equation} \label{eq:2_dim_comp}	
\begin{split}
S_{\widehat{g}}(g)
=
S_{\widehat{g}}(\overline{g})
&=\lim_{R\to\infty}	
\left(\int_{B_{R}}(\scal_{\overline{g}}+2)\dv_{\overline{g}}-2R_{\widehat{g}}(\overline{g},R)\right)\\	
&=
\lim_{R\to\infty}\left(\int_{B_{R}}\scal_{\overline{g}}\dv_{\overline{g}}+2\int_{\widehat{B}_{R}}\dv_{\widehat{g}}\right)\\	
&=
\int_{B_{R_0}}\scal_{\overline{g}} \, \dv_{\overline{g}}
-\int_{\widehat{B}_{R_0}}\scal_{\widehat{g}} \, \dv_{\widehat{g}}.	
\end{split}	
\end{equation}	
Now replace the closed set $\widehat{M}\setminus\widehat{B}_{R_0}$ by its one-point compactification at infinity, that is, a closed disk $D$, and choose a metric $\widetilde{g}$ on $D\cup (\widehat{B}_{R_0}\setminus \widehat{B}_{R_0-\epsilon})$ which agrees with $\widehat{g}$ on $\widehat{B}_{R_0}\setminus \widehat{B}_{R_0-\epsilon}$. Then $\widehat{g}$ extends to a metric $\widehat{g}_1$ on $D\cup \widehat{B}_{R_0}\cong S^2$, and to a metric $\overline{g}_1$ on $D\cup_{\varphi}B_{R_0}\cong \overline{M}$. Note that $\widehat{g}_1$ has a conical singularity with the same angle defect as $\widehat{g}$. The Gauss--Bonnet theorem for compact conical surfaces (see, for example, \cite[Proposition~1]{Troyanov1991}) yields	
\begin{align*}	
\int_{S^2}\scal_{\widehat{g}_1}\dv_{\widehat{g}_1} 
= 4\pi\chi(S^2)+2(\omega-2\pi).	
\end{align*}	
Using the fact that $\widehat{g}_1$ and $\overline{g}_1$ agree on $D$, we obtain	
\begin{align*}	
4\pi(\chi(\overline{M})-2)+2(2\pi-\omega)
&=
4\pi\chi(\overline{M})-(4\pi\chi(S^2)+2(\omega-2\pi))\\	
&=
\int_{\overline{M}}\scal_{\overline{g}_1}\dv_{\overline{g}_1}	
-	
\int_{S^2}\scal_{\widehat{g}_1}\dv_{\widehat{g}_1}\\	
&=
\int_{D}	\scal_{\overline{g}_1}\dv_{\overline{g}_1}	
+\int_{B_R}\scal_{\overline{g}_1}\dv_{\overline{g}_1}\\	
&\qquad  
-\int_{D} \scal_{\widehat{g}_1}\dv_{\widehat{g}_1}
-\int_{\widehat{B}_{R}}\scal_{\widehat{g}_1}\dv_{\widehat{g}_1}\\	
&=	
\int_{B_R}\scal_{\overline{g}}\dv_{\overline{g}}	
-\int_{\widehat{B}_{R}}\scal_{\widehat{g}}\dv_{\widehat{g}} ,
\end{align*}	
and combining this with \eqref{eq:2_dim_comp} yields \eqref{eq:S_dim_two}. The identity \eqref{eq:m_dim_two} is immediate. The inequality \eqref{eq:pmt_dim_two} follows easily from the assumption $\scal_g\geq -2$ and the well-known identity $\chi(\overline{M})\leq  \chi(S^2)=2$. In case of equality in \eqref{eq:pmt_dim_two}, $\scal_g\equiv -2$, so $g$ is a metric of constant curvature $-1$, and $\overline{M}$ is diffeomorphic to $S^2$, so $M$ is diffeomorphic to $\R^2$. Since $g$ is asymptotic to $\hat{g}$ and both are of constant curvature, they must agree up to isometry in a neighborhood of infinity. Since both metrics are of constant curvature and they are defined on diffeomorphic manifolds, they must be isometric.
\end{proof}
\begin{rem}
An analogue of this formula has been established for asymptotically conical surfaces, see for example \cite[Sec.~1.1.1]{Chr13}.
\end{rem}

\subsection{A conformal positive mass theorem}\label{subsec:conformal_PMT}

%

Let $\hg$ be an APE metric of constant scalar curvature $\scal_{\hg}=-n(n-1)$ on the manifold $M$. Let $\delta>\frac{n-1}{2}$ and consider the set
\begin{align}\label{eq:constantscalarcurvaturemetrics}
\mathcal{C}
=\left\{g\in \mathcal{R}^{k,\alpha}_{\delta}(M,\hg)\mid \scal_g=-n(n-1)\right\}
\end{align}
of constant scalar curvature metrics asymptotic to $\hg$.

\begin{prop}\label{prop_conformal_decomposition}
For each $g\in \mathcal{R}^{k,\alpha}_{\delta}(M,\hg)$, there exists a unique function $w\in C^{k,\alpha}_{\delta}$ such that $\overline{g}=e^{2w}g\in \mathcal{C}$. Moreover, the set $\mathcal{C}$ is an analytic manifold and
the map 
\begin{align*}
\Phi: C^{k,\alpha}_{\delta}(M)\times \mathcal{C}&\to \mathcal{R}^{k,\alpha}_{\delta}(M,\hg) ,\\
(w,g)&\mapsto e^{2w}g,
\end{align*}
is a diffeomorphism of Banach manifolds.
\end{prop}

\begin{proof}
The first assertion follows from the resolution of the Yamabe problem in the AH setting as formulated in \cite[Theorem 1.7]{AllenIsenbergLee2018}, except the precise decay of the conformal factor as formulated in this proposition.
According to \cite[Theorem 1.7]{AllenIsenbergLee2018}, there is a function $w \in C^{k,\alpha}_1$ so that the metric $\overline{g}=e^{2w}g$ has constant scalar curvature $-n(n-1)$.
The function $w$ satisfies the Yamabe equation
 \begin{align*}
-e^{2w}(n-1)n=e^{2w}\scal_{\overline{g}}=\scal_g+2(n-1)\Delta_gw-(n-1)(n-2)|dw|^2_g.
 \end{align*}
In the following we are going to use this equation to show that $w\in  C^{k,\alpha}_{\delta}$. We rewrite the above equation as
\begin{equation}\begin{split}\label{eq:Yamabe_equation}
2(n-1)(\Delta w +n\omega)&=(n-1)(n-2)|dw|^2_g-(\scal_g+n(n-1))\\
&\qquad-n(n-1)F(w),
\end{split}
\end{equation} 
where $F(x)=e^{2x}-1-2x$.
By assumption, $g-\hat{g}\in C^{k,\alpha}_{\delta}$, so that  we have $\scal_g+n(n-1)\in C^{k-2,\alpha}_{\delta}$. From $w \in C^{k,\alpha}_1$ and because $F(x)=O(x^2)$ as $x\to 0$, we get 
\begin{align*}
(n-1)(n-2)|dw|^2_g-(\scal_g+n(n-1))
-n(n-1)F(w)\in C^{k,\alpha}_{\min\left\{2,\delta\right\}}.
\end{align*}
By \cite[Theorem C and Proposition E]{Lee06}, we know that the operator 
\begin{align*}\Delta  +n:C^{k,\alpha}_{\eta}\to C^{k-2,\alpha}_{\eta}
\end{align*}
 is an isomorphism for  $\eta\in (-1,n)$. Thus we obtain $w\in C^{k,\alpha}_{\min\left\{2,\delta\right\}}$.
This implies
 \begin{align*}
2|dw|^2_g-(\scal_g+6)-6F(w)\in C^{k-2,\alpha}_{\min\left\{4,\delta\right\}},
\end{align*}
and therefore, $w \in C^{k,\alpha}_{\min\left\{4,\delta\right\}}$. Repeating this procedure a finite number of times yields $w \in C^{k,\alpha}_{\delta}$, as desired.

To show that $\mathcal{C}$ is a manifold, consider the analytic map 
\begin{align*}
\Psi:\mathcal{R}^{k,\alpha}_{\delta}(M,\hg) &\to C^{k-2,\alpha}_{\delta}(M),\\
\Psi:g&\mapsto \scal_g+n(n-1).
\end{align*}
The differential of this map is given by
\begin{align*}
D_g\Psi(h)=D_g\scal(h)=\Delta_g(\trace h)+\mathrm{div}_g(\mathrm{div}_g h)-\langle \ric_g,h\rangle_g.
\end{align*}
In particular, if $g\in \mathcal{C}$ and $f\in C^{k,\alpha}_{\delta}$,
\begin{align*}
D_g\Psi(f g)=(n-1)(\Delta_{g} f+nf).
\end{align*}
Due to \cite[Theorem C and Proposition E]{Lee06}, the operator
\begin{align*}
\Delta_g+n: C^{k,\alpha}_{\delta}(M)\to C^{k-2,\alpha}_{\delta}(M)
\end{align*}
is an isomorphism.
Thus, $D_g\Psi$ is surjective and hence, $\mathcal{C}$ is an analytic manifold with $D_g\Psi=T_g\mathcal{C}$ by the implicit function theorem for Banach manifolds.

It remains to show that the map $\Phi$ is a diffeomorphism. We already know  that $\Phi$ is smooth and bijective. Thus we are done, if we can show that $\Phi$ is a local diffeomorphism in a neighborhood of each tuple 
$(w,g)\in  C^{k,\alpha}_{\delta}(M)\times \mathcal{C}$.
At first we have
\begin{align}\label{eq:conformal_splitting}
C^{k,\alpha}_{\delta}(S^2T^*M)=C^{k,\alpha}_{\delta}(M)\oplus T_g\mathcal{C}.
\end{align}
for each $g\in \mathcal{C}$: For $h\in C^{k,\alpha}_{\delta}(S^2T^*M)$, there exists a unique function $w\in C^{k,\alpha}_{\delta}(M,\hg)$ solving the equation 
\begin{align*}D_g\Psi(h)=D_g\Phi(f g)=(n-1)(\Delta_{g} w+nw)
\end{align*}
 and we obtain \eqref{eq:conformal_splitting} by writing $h=w g+(h-w g)$. The differential
\begin{align*}
D_{(0,g)}\Phi: C^{k,\alpha}_{\delta}(M)\oplus T_g\mathcal{C}\to C^{k,\alpha}_{\delta}(S^2T^*M)
\end{align*}
directly corresponds to the splitting \eqref{eq:conformal_splitting} and is therefore an isomorphism. Thus, $\Phi$ is a local diffeomorphism around each tuple $(1,g)$ where $g\in \mathcal{C}$. For a general tuple $(w_0,g)\in C^{k,\alpha}_{\delta}(M)\times \mathcal{C}$, we write $\Phi$ is the composition of maps
\begin{align*}
(w,g)\mapsto (w-w_0,g)\mapsto \Phi(w-w_0,g)=e^{2w}e^{-2w_0}g\mapsto e^{2w}g=\Phi(w,g).
\end{align*}
We see that $\Phi$ is a local diffeomorphism around $(w_0,g)$ since $\Phi$ is a local diffeomorphism around $(1,g)$ and the maps $w\mapsto w-w_0$ and $g\mapsto e^{2w_0}g$ are obviously diffeomorphisms. This finishes the proof.
\end{proof}

\begin{cor}\label{cor:crit_points_mass}
The functional $g\mapsto \m_{\rm{VR},\widehat{g}}(g)$ is an analytic functional on the analytic manifold $\mathcal{C}$. Furthermore, $g\in \mathcal{C}$ is a critical point of $\m_{\rm{VR},\widehat{g}}$ if and only if $\ric_g=-(n-1)g$.
\end{cor}
\begin{proof}
We know that $g\mapsto S_{\hat{g}}(g)$ is an analytic functional on $\mathcal{R}^{k,\alpha}_{\delta}(\widehat{M},\hg)$ which by definition agrees with $g\mapsto -m_{\rm{VR},\widehat{g}}(g)$ on $\mathcal{C}$ and the first assertion follows. 
For the second assertion, Proposition \ref{prop:first_variation_EHaction} tells us that the gradient of $S_{\hg}$ is
\begin{align*}
\grad S_{\hg}(g)=
\ric_ g-\frac{1}{2}\scal_g\cdot g-\frac{1}{2}(n-1)(n-2)g.
\end{align*}
In particular, $\grad S_{\hg}(g)$ is trace-free if $\scal_g=-n(n-1)$.
Thus, any $g\in \mathcal{C}$ is a critical point of $S$ with respect to conformal variations. By \eqref{eq:conformal_splitting}, we therefore have $\ric_ g=-(n-1)g$ for $g\in \mathcal{C}$ if and only if it is a critical point of $S_{\hg}|_{\mathcal{C}}=-\m_{\rm{VR},\widehat{g}}|_{\mathcal{C}}$. This proves the assertion.
\end{proof}
 We next turn to the case where $g$ is conformal to $\hg$, with $\scal_{\hg}=-n(n-1)$. 

\begin{thm}\label{thm:conformal_positive_mass}
Let $(M^n,\hg)$ be a complete APE manifold with $\scal_{\hg}=-n(n-1)$. Let $w\in C^{k,\alpha}_{\delta}(M)$ for some $\delta>\frac{n-1}{2}$ and $k+\alpha\geq\delta$, and 
consider the conformal metric $g=e^{2w}\hg$ on $M$.\\
Then if $\scal_{g}\geq -n(n-1)$ and $\scal_g+n(n-1)\in L^1$, we have $\mathfrak{m}_{\rm{VR},\hg}(g) \geq 0$.
If furthermore $\mathfrak{m}_{\rm{VR},\hg}(g) = 0$, then $g = \hg$.
\end{thm}

\begin{proof}
We assume that the diffeomorphism $\varphi$ for defining the mass is $\varphi = \identity$.
We substitute $\phi=e^{\frac{1}{2}(n-2)w}$ so that $g=\phi^{4/(n-2)}\hg$. Note that $\phi-1\in C^{k,\alpha}_{\delta}$. 
We have
\begin{align*}
\mathfrak{m}_{{\rm ADM},\widehat{g}}(g,R)
&= \int_{\partial B_R}({\mathrm{div}}_{\hg}(g)-d{\trace}_{\hg}(g))(\nu)\dv_{\widehat{g}}\\ 
&= \int_{\partial B_R} (1-n) d(\phi^{4/(n-2)})(\nu)\dv_{\widehat{g}}\\ 
&= -\frac{4(n-1)}{n-2} \int_{\partial B_R} \phi^{(6-n)/(n-2)} d\phi(\nu)\dv_{\widehat{g}}
\end{align*}
and  
\begin{align*}
RV_{\widehat{g}}(g,R)
&= \int_{B_R}\dv_g - \int_{B_{R}}\dv_{\widehat{g}} 
= 
\int_{B_R} \left( \phi^{2n/(n-2)} - 1 \right) \dv_{\widehat{g}}
\end{align*}
so that
\begin{equation}
\begin{split}
\label{eq:conformal_mass}
\mathfrak{m}_{\rm{VR},\hg}(g)
&=
\lim_{R \to \infty}
\left( m_{{\rm ADM},\widehat{g}}(g,R)+2(n-1) RV_{\widehat{g}}(g,R) \right) \\
&=
\lim_{R \to \infty} \Bigg(
-\frac{4(n-1)}{n-2} \int_{\partial B_R} d\phi(\nu)\dv_{\widehat{g}} \\
&\qquad\qquad
-\frac{4(n-1)}{n-2} \int_{\partial B_R} \left(\phi^{(6-n)/(n-2)} - 1\right) d\phi(\nu)\dv_{\widehat{g}}\\
&\qquad\qquad
+ 2(n-1) \int_{B_{R}} \left( \phi^{2n/(n-2)} - 1 \right) \dv_{\widehat{g}}
\Bigg) .
\end{split}
\end{equation}
With the decay of $\phi$, the first and last terms in this expression may be infinite. However, we will see below that they combine to yield a finite quantity with the correct sign. Furthermore, the integrand of the middle term satisfies $\left(\phi^{(6-n)/(n-2)} - 1\right) d\phi(\nu)\in O(e^{-2\delta r})$ and since $2\delta>n-1$, this term vanishes in the limit.

To handle the first and last terms in \eqref{eq:conformal_mass}, we need to establish that $\phi \geq 1$. 
In order to do this, recall that the scalar curvatures of $g$ and $\hg$ are related by
\begin{equation*}
e^{2w}\scal_g=\scal_{\hg}+2(n-1){\Delta}_{\hg}w-(n-2)(n-1)|d w|_{\hg}^2,
\end{equation*}
from which we have
\begin{equation}\label{eq-confchange}
2{\Delta}_{\hg}w
=\frac{1}{(n-1)}\left( e^{2w} \scal_g-\scal_{\hg}+(n-1)(n-2)|d w|_{\hg}^2 \right).
\end{equation}
Using the fact that $\scal_g \geq \scal_{\hg}=-n(n-1)$ we find that
\begin{equation*}
2{\Delta}_{\hg}w \geq -n(e^{2w}-1) + (n-2)|d w|_{\hg}^2.
\end{equation*}
For the sake of contradiction assume that $w<0$ somewhere on $M$. Since we know that $w\to0$ at infinity, there must be a minimum attained at some  point $x_0$ where $w<0$ and therefore $e^{2w}-1<0$ at $x_0$. In particular, this implies ${\Delta}_{\hg}w>0$ at $x_0$, which by the maximum principle contradicts the fact that $x_0$ was a minimum.
We therefore conclude that $w\geq0$ and $\phi \geq 1$.

Applying the divergence theorem to \eqref{eq:conformal_mass} and using the fact that the middle integral vanishes in the limit, we see that
\[ \begin{split}
\mathfrak{m}_{\rm{VR},\hg}(g)&=\lim_{R \to \infty} 
2(n-1) \int_{{B}_{R}}
\left( \phi^{2n/(n-2)}-1+\frac{2}{n-2}{\Delta}_{\hg}\phi \right) \dv_{\widehat{g}}\\
&=2(n-1) \int_{M}
\left( \phi^{2n/(n-2)}-1+\frac{2}{n-2}{\Delta}_{\hg}\phi \right) \dv_{\widehat{g}}
\end{split} \]
In terms of $\phi$, the scalar curvatures of $g$ and $\hg$ are related by
\begin{equation}\label{eq:conformal_scal}
\begin{split}
\frac{4(n-1)}{n-2}{\Delta}_{\hg}\phi
&= \scal_g\phi^{(n+2)/(n-2)}-\scal_{\hg}\phi \\ 
&= n(n-1) \left( \phi -\phi^{(n+2)/(n-2)} \right)\\
&\qquad +
 \left(\scal_g+(n-1)n\right)\phi^{(n+2)(n-2)}
\end{split}
\end{equation}
from which we find that
\begin{align*}
\mathfrak{m}_{\rm{VR},\hg}(g)&=2(n-1)\int_{M} \left( \phi^{2n/(n-2)}-1+\frac{2}{n-2}{\Delta}_{\hg}\phi \right) \dv_{\widehat{g}}\\
&=\int_{M}[2(n-1)F(\phi)+(\scal_g+n(n-1))\phi^{(n+2)/(n-2)}]\dv_{\hg},
\end{align*}
where the smooth function $F:(0,\infty)\to \R$ is defined by
\begin{align}\label{eq:conformal_auxialliary_function}
F(x) = x^{2n/(n-2)}-1+\frac{n}{2}x-\frac{n}{2}x^{(n+2)/(n-2)}.
\end{align}
Recall that $\phi \geq 1$, so we are done with the proof if we can show $F(x)\geq 0$ for $x\geq 1$ with $F(x)=0$ only if $x=1$. 
In fact, a direct computation shows that $F(1)=F'(1)=0$ and
\begin{align*}
F''(x) 
= \frac{2n(n+2)}{(n-2)^2}x^{(6-n)/(n-2)}\left( x-1 \right),
\end{align*}
hence $F''(1)=0$ and $F''(x)>0$ for $x>1$
and the desired statement follows.
\end{proof}

\begin{rem}Since every rotationally symmetric AH metric on $\R^n$ is conformal to the hyperbolic metric, we also get a positive mass theorem for such metrics.
\end{rem}

\begin{rem}
An analogue of Theorem \ref{thm:conformal_positive_mass} is well known for the ADM mass of asymptotically Euclidean manifolds, {(see, for example, \cite[Lem.~3.3]{SchoenYau1979})}.
\end{rem}

\subsection{Positive mass theorem for AH metrics on $\mathbb R^3$}

Now we consider AH metrics on $\mathbb R^3$. In this case, we are able to prove a positive mass theorem for the volume-renormalized mass.


\begin{thm}\label{thm:3dpmt}
Let $(\R^3,g)$ be an APE manifold with $g\in
 \mathcal{R}^{k,\alpha}_{\delta}(\R^3,g_{\mathrm{hyp}})$ for some $\delta>1$, satisfying $\scal_g\geq -6$ {and $\scal_g+6\in L^1$}. Then, $\m_{{\rm VR},g_{\mathrm{hyp}}}(g)\geq 0$ and equality holds if and only if $g$ is isometric to $g_{\mathrm{hyp}}$.
\end{thm}

\begin{proof}
Let $\mathcal{C}$ be as in \eqref{eq:constantscalarcurvaturemetrics}, with $M=\R^3$ and $\hg=g_{\mathrm{hyp}}$.
By Proposition \ref{prop_conformal_decomposition}, there is a function $w\in C^{k,\alpha}_{\delta}(\R^3,g_{\mathrm{hyp}})$ such that $\overline{g}=e^{2\omega}g\in\mathcal{C}$.
From Theorem \ref{thm:conformal_positive_mass}, it follows that $\m_{{\rm VR},\overline{g}}(g)\geq 0$, or equivalently $\m_{{\rm VR},g_{\mathrm{hyp}}}(g)\geq m_{{\rm VR},g_{\mathrm{hyp}}}(\overline{g})$ with equality if and only if $g=\overline{g}$.

It thus remains to show that $\m_{{\rm VR},g_{\mathrm{hyp}}}(\overline{g})\geq 0$, with equality if and only if $\overline{g}$ is isometric to $g_{\mathrm{hyp}}$.
We know that $\overline{g}-g_{\mathrm{hyp}}\in C^{k,\alpha}_{\delta}(S^2T^*\R^3,g_{\mathrm{hyp}})$.
Since $C^{\infty}_{c}$ is dense in $C^{k,\alpha}_{\delta}$, we can find a sequence $\{g_i\}_{i\in\N}$ of AH metrics on $\R^3$ such that $g_i-g_{\mathrm{hyp}}$ is compactly supported for each $i$ and $g_i\to \overline{g}$ with respect to the $C^{k,\alpha}_{\delta}$-norm as $i\to\infty $.
 By Proposition \ref{prop_conformal_decomposition}, we get a sequence
 of constant scalar curvature metrics $\overline{g}_i=e^{2w_i}g_i$, where the sequence of conformal factors $w_i\in C^{k,\alpha}_{\delta}$ converges to $0$ in $C^{k,\alpha}_{\delta}$. 
By \eqref{eq:Yamabe_equation}, we have
\begin{align*}
4(\Delta_{g_i}w_i +3w_i)=2|dw_i|^2_{g_i}-(\scal_{g_i}+6)-6F(w_i),
\end{align*}
where $F(x)=e^{2x}-1-2x$.
Recall that $\scal_{g_i}+6$ is compactly supported, so that $w_i \in C^{k,\alpha}_{\beta}$ for all $\beta>0$. We now repeat the above argument to improve the decay rate of the functions $w_i$. First, $|dw_i|^2_{g_i}-(\scal_{g_i}+6)-6F(w_i)\in C^{k-2,\alpha}_{2\delta}$, so  by the isomorphism property of $\Delta+3$, we have $w_i\in C^{k-2,\alpha}_{\min\left\{2\delta,3\right\}}$. Repeating this step twice, we see that $|dw_i|^2_{g_i}-(\scal_{g_i}+6)-6F(\omega_i)\in  C^{k-2,\alpha}_{\beta}$ and $w_i\in H^{k}_{\beta}$ for any $\beta< 3$. Therefore, $w_i\in O(e^{-{\beta r}})$ for any $\beta> -3$ and $\overline{g}_i-g_{\mathrm{hyp}}\in C^{k,\alpha}_{\beta}$ for any  $\beta< 3$. For this reason, $\m_{{\rm ADM},\widehat{g}}(g,R)\to 0$ as $R\to\infty$, and
\begin{align*}
\m_{{\rm VR},g_{\mathrm{hyp}}}(\overline{g}_i)
=2(n-1)RV_{g_{\mathrm{hyp}}}(\overline{g}_i).
\end{align*}
{Furthermore, we know by the work of Brendle and Chodosh \cite{BrendleChodosh2014} (see also \cite[Sec.~5]{Chodosh2016})} that $RV_{g_{\mathrm{hyp}}}(\overline{g}_i)\geq0$. Since $\overline{g}_i$ is a sequence of constant scalar curvature metrics converging to $\overline{g}$ in $C^{k,\alpha}_{\delta}$, we get from Corollary \ref{cor:crit_points_mass} that
\begin{align*}
2(n-1)RV_{g_{\mathrm{hyp}}}(\overline{g}_i)=\m_{{\rm VR},g_{\mathrm{hyp}}}(\overline{g}_i)\to \m_{{\rm VR},g_{\mathrm{hyp}}}(\overline{g}).
\end{align*}
Consequently, $\m_{{\rm VR},g_{\mathrm{hyp}}}(\overline{g})\geq 0$. To finish the proof, it remains to consider the equality case. Suppose that $\m_{{\rm VR},g_{\mathrm{hyp}}}(\overline{g})= 0$. Then, $\overline{g}$ is a critical point of 
\begin{align*}
\mathcal{C}\ni g\mapsto S_{g_{\mathrm{hyp}}}(g)= - \m_{{\rm VR},g_{\mathrm{hyp}}}(g).
\end{align*}
By Corollary \ref{cor:crit_points_mass}, it follows that $\overline{g}$ is a PE metric on $\R^3$.
 Since $n=3$, metric $\overline{g}$ has constant curvature and must therefore be isometric to $g_{\mathrm{hyp}}$.
\end{proof}

\section{A renormalized expander entropy for AH manifolds}\label{sec:entropy}

In this section we will define and study a renormalized expander entropy for AH manifolds, which turns out to be a monotone quantity for the Ricci flow. Throughout, let $(M,g)$ and $(\widehat M,\hg)$ be APE manifolds with isometric conformal boundaries. Assume that $(M,g)$ is complete and that $(\widehat M,\hg)$ has integrable normalized scalar curvature.
%

\subsection{Definition of the entropy}
In analogy to the expander entropy for compact manifolds defined in \cite{FIN05}, we consider the following definition.
\begin{lem}\label{lem_W_entropy}
For
$f\in C^{\infty}_c(M)$, the limit
\begin{align*}
\mathcal{W}_{\rm{AH},\hg}(g,f) & 
=\lim_{R \to \infty}\Big(\int_{B_R} \Big(
\left( |\nabla f|^2+\scal_g-2(n-1)f \right)e^{-f}\Big) \dv_g\\
&\qquad + (n-2)(n-1)\int_{ B_R}e^{-f}\dv_g\\
&\qquad +
2(n-1)\int_{\hat{B}_R}\dv_{\hat{g}}- m_{{\rm ADM},\widehat{g}}(g,R)\Big)\\
&=\lim_{R \to \infty}\Big(\int_{B_R} \Big(
\left( |\nabla f|^2+\scal_g+f \right)e^{-f}\dv_g\\
&\qquad\qquad -2(n-1) \left( (f+1)e^{-f} - 1 \right)
\Big) \dv_g\\
&\qquad\qquad - m_{{\rm ADM},\widehat{g}}(g,R)
 - 2(n-1)RV_{\widehat{g}}(g,R)\Big).
\end{align*}
is finite.
\end{lem}

\begin{proof}
Substitute $e^{-f}=\omega^2$ and define
\begin{align*}
\widetilde{\mathcal{W}}_{\rm{AH},\hg}(g,\omega)=\mathcal{W}_{\rm{AH},\hg}(g,f).
\end{align*}
Then we get
\begin{align*}
\widetilde{\mathcal{W}}_{\rm{AH},\hg}(g,\omega) 
&= 
\lim_{R \to \infty}
\Big( \int_{B_R}\Big( 4|\nabla \omega|^2 +(\scal_g+n(n-1))\omega^2\\
&\qquad +2(n-1)[(\log(\omega^2)-1)\omega^2+1]
\Big) \dv_g - RV_{\widehat{g}}(g,R) \Big).
\end{align*}
Substitute further $u= \omega - 1$ and let
\begin{align*}
\overline{\mathcal{W}}_{\rm{AH},\hg}(g,u)=\widetilde{\mathcal{W}}_{\rm{AH},\hg}(g,\omega).
\end{align*}
In addition, set $G(x)=2(n-1) \left( (\log((x+1)^2)-1)(x+1)^2+1 \right)$. Note that $G$ is nonnegative and $G(0)=0$. Then,
\begin{align*}
\overline{\mathcal{W}}_{\rm{AH},\hg}(g,u) &=\lim_{R \to \infty}\Big(\int_{ B_R} \left( 4|\nabla u|^2+(\scal_g +n(n-1))(u+1)^2 + G(u) \right)\dv_g\\
&\qquad\qquad -\left( m_{{\rm ADM},\widehat{g}}(g,R)+2(n-1)RV_{\widehat{g}}(g,R) \right)\Big) \\
&=
\int_M \Big( 4|\nabla u|^2+(\scal_g +n(n-1))u^2\\
&\qquad +2(\scal_g +n(n-1))u+G(u) \Big) \dv_g 
+ S_{\hg}(g),
\end{align*}
where $ S_{\hg}(g)$ is well-defined and finite by Theorem \ref{thm-well-def-mass}. Since $f\in C^{\infty}_c(M)$ we have $\omega-1\in C^{\infty}_c(M)$ so $u\in C^{\infty}_c(M)$ and $G(u)\in C^{\infty}_c(M)$, and the integral is finite, completing the proof.
\end{proof}

\begin{defn}
The {\em renormalized expander entropy} of $(M,g)$ is defined as
\begin{align*}
\mu_{\rm{AH},\hg}(g)
=\inf_{ f\in C_c^{\infty}(M)}\mathcal{W}_{\rm{AH},\hg}(g,f)
=\inf_{ u\in C_c^{\infty}(M)}\overline{\mathcal{W}}_{\rm{AH},\hg}(g,u).
\end{align*}
\end{defn}
%

\begin{rem}
The proof of Lemma \ref{lem_W_entropy} implies that
\begin{align*}
\mu_{\rm{AH},\hg}(g)&=S_{\hg}(g)+\inf_{u\in C^{\infty}_c(M)}\int_M \Big( 4|\nabla u|^2+(\scal_g +n(n-1))u^2\\
&\qquad\qquad\qquad\qquad\qquad +2(\scal_g +n(n-1))u+G(u) \Big) \dv_g  ,
\end{align*}
where $G(x)=2(n-1) \left( (\log((x+1)^2)-1)(x+1)^2+1 \right)$.
We obtain an additivity of the functional $\mu_{\rm{AH},\hg}$. Suppose that $(\overline{M},\tilde{g})$ is another APE manifold with integrable normalized scalar curvature whose conformal boundary is isometric to the one of $(M,g)$. Then the additivity of the renormalized Einstein--Hilbert action in Proposition \ref{cor:additivity_renormalized_EH}  gives
\begin{align*}
\mu_{\rm{AH},\tilde{g}}(g)=\mu_{\rm{AH},\hg}(g)+m_{\rm{VR},\hg}(\tilde{g}).
\end{align*}
Thus, the variational properties of the functional $g\mapsto \mu_{\rm{AH},\hg}(g)$ are independent of the choice of the reference metric $\hg$.
\end{rem}

\subsection{Existence of minimizers}\label{subsec:compact_exhaustion}
We are going to show that the infimum in the definition $\mu_{\rm{AH},\hg}(g)$ is always achieved by a unique $C^{k,\alpha}$-function $f_g$ and that $f_g$ and hence also $\mu_{\rm{AH},\hg}(g)$ depend analytically on the metric $g$.

\begin{lem}\label{lem_H1-comparison}
For sufficiently small $\epsilon>0$ there are positive constants $C_{\epsilon}$ and $C$, depending on $g$ and $\hg$, such that
\begin{align*}
C\left\|u\right\|_{H^1}^2 - C
\leq 
\overline{\mathcal{W}}_{\rm{AH},\hg}(g,u)- S_{\hg}(g)
\leq 
C_{\epsilon} (1+\left\|u\right\|_{H^1}^{\epsilon}) \left\|u\right\|_{H^1}^2 + C
\end{align*}
for all $u\in C^{\infty}_c(M)$. In particular, 
\begin{align*}
\mu_{\rm{AH},\hg}(g)
= \inf_{ u\in H^1(M)}\overline{\mathcal{W}}_{\rm{AH},\hg}(g,u)
\end{align*}
and $\mu_{\rm{AH},\hg}(g)>-\infty$.
\end{lem}

\begin{proof}
Recall that $G(x)=2(n-1) \left( (\log((x+1)^2)-1)(x+1)^2+1 \right)$.
This function satisfies $G(x) \geq 0$ for all $x\in\R$ and $G(x)=0$ if and only if $x=0$ or $x=-2$. Taylor expansion at $x=0$ shows that $G(x)=4(n-1)x^2+O(|x|^3)$ as $x\to 0$, whereas  $G(x)=O(\log(|x|)x^2)$ as $|x|\to\infty$. Summarizing these estimates, we get for any given $\epsilon>0$ a constant $C_{\epsilon}>0$ such that
\begin{align}\label{eq:estimate_G}
G(x)\leq 4(n-1)x^2+C_{\epsilon}|x|^{2+\epsilon}.
\end{align}
Therefore, we have
\begin{align*}
&\overline{\mathcal{W}}_{\rm{AH},\hg}(g,u) - S_{\hg}(g) \\
&\qquad=
\int_M \left( 4|\nabla u|^2+(\scal_g +n(n-1))u^2 +2(\scal_g +n(n-1))u+G(u) \right) \dv_g \\
&\qquad \leq 
\int_M 
\left( C(|\nabla u|^2+u^2)+(\scal_g +n(n-1))^2 + C_{\epsilon} u^{2+\epsilon} \right) \dv_g \\
&\qquad \leq
C_{\epsilon} (1+\left\|u\right\|_{H^1}^{\epsilon}) \left\|u\right\|_{H^1}^2
+\left\| \scal_g +n(n-1)\right\|_{L^2}^2.
\end{align*}
Here, we used \eqref{eq:estimate_G} and the Peter--Paul inequality in the first inequality. In the second one, we used that $L^{2+\epsilon}\subset H^1$ for sufficiently small $\epsilon>0$ and
$\scal_{g}+n(n-1)\in C^{k-2,\alpha}_{\delta}\subset L^2$ since $\delta>\frac{n-2}{2}$. This gives us the upper bound.

Now let us find the lower bound. From \cite[Proposition~F]{Lee06} we know that there is a gap around zero in the essential spectrum of the scalar Laplacian $\Delta$. Since there is also no zero eigenvalue, we get a lower bound $\Delta \geq \Lambda > 0$. Therefore, again by the Peter--Paul inequality,
\begin{align*}
&\overline{\mathcal{W}}_{\rm{AH},\hg}(g,u) - S_{\hg}(g)\\
&\qquad=
\int_M \left( 4|\nabla u|^2+(\scal_g +n(n-1))u^2 + 2(\scal_g +n(n-1))u+G(u) \right) \dv_g \\
&\qquad\geq
\int_M \left( 2|\nabla u|^2 + 2 \Lambda u^2+(\scal_g +n(n-1))u^2  + G(u) \right) \dv_g \\
&\qquad \qquad 
-\int_M \left( \epsilon u^2+\frac{1}{\epsilon}(\scal_g +n(n-1))^2 \right) \dv_g \\
&\qquad 
=\int_M \left( 2|\nabla u|^2 + \left(2\Lambda + \scal_g +n(n-1) - \epsilon \right) u^2 + G(u)\right) \dv_g \\
&\qquad \qquad 
-\frac{1}{\epsilon} \left\| \scal_g +n(n-1)\right\|_{L^2}^2
\end{align*}
for any $\epsilon>0$. Since $\scal_g +n(n-1)\to 0$ at infinity, the set
\begin{align*}
K=\left\{x\in M \mid  2 \Lambda + \scal_g +n(n-1) - \epsilon \leq 0\right\}
\end{align*}
is compact, provided that we chose $\epsilon>0$ such that $\epsilon < 2\Lambda$. Since $G$ is a nonnegative function and since the scalar curvature is bounded from below, there exists a constant $C>0$ such that
\begin{align*}
\overline{\mathcal{W}}_{\rm{AH},\hg}(g,u) - S_{\hg}(g)
&\geq 2
\int_M|\nabla u|^2\dv_g+\int_K (G(u)-C u^2)\dv_g  \\
&\qquad
-\frac{1}{\epsilon} \left\| \scal_g +n(n-1)\right\|_{L^2}^2
\end{align*}
Since the function $G(x)$ has faster than quadratic growth, 
we have $ G(x) - C x^2 \geq -C_1$ for some constant $C_1 > 0$. Thus,
\begin{align*}
&\overline{\mathcal{W}}_{\rm{AH},\hg}(g,u) - S_{\hg}(g) \\
&\qquad \geq 2
\int_M|\nabla u|^2\dv_g - C_1 \volume(K,g)
-\frac{1}{\epsilon}\left\| \scal_g +n(n-1)\right\|_{L^2}^2\\
&\qquad \geq 
\int_M
\left( |\nabla u|^2 + \Lambda u^2 \right) \dv_g  
- C_1 \volume(K,g)
-\frac{1}{\epsilon}\left\| \scal_g +n(n-1)\right\|_{L^2}^2\\
&\qquad \geq 
C_2 \left\|u\right\|_{H^1}^2 
- C_1 \volume(K,g)
-\frac{1}{\epsilon}\left\| \scal_g +n(n-1)\right\|_{L^2}^2.
\end{align*}
This proves the desired lower bound.
\end{proof}

\begin{prop} \label{prop:EL_equation_entropy}
The Euler--Lagrange equation of the functional $\mathcal{W}_{\rm{AH},\hg}(g,f)$ is
\begin{align} \label{eq:EL_equation_entropy}
2\Delta f+|\nabla f|^2-\scal_g-n(n-1)+2(n-1)f=0
\end{align}
\end{prop}

\begin{proof}
For $v\in C^{\infty}_c(M)$ we get
\begin{align*}
&\frac{d}{dt} \mathcal{W}_{\rm{AH},\hg}(g,f+tv)|_{t=0} \\
&\qquad=
\frac{d}{dt}\int_M 
\left( (|\nabla f|^2+\scal_g +n(n-1)) e^{-f} -2(n-1)( (f+1)e^{-f}-1) \right)
\dv |_{t=0} \\
&\qquad=
\int_M
\left( 
2\langle \nabla f,\nabla v\rangle  -v \left( |\nabla f|^2 + \scal_g + n(n-1) \right)
\right)
e^{-f}\\
&\qquad\qquad 
-2(n-1) \int_M (ve^{-f}-(f+1)ve^{-f})\dv\\
&\qquad=
\int_M \left( 2\Delta f+|\nabla f|^2- \scal_g-n(n-1)+2(n-1)f \right) ve^{-f}\dv,
\end{align*}
where we used integration by parts in the last equality.
\end{proof}

Next, we are going to discuss existence and uniqueness of solutions of the equation \eqref{eq:EL_equation_entropy} with suitable conditions at infinity.
We begin by proving uniqueness of solutions.

\begin{lem} \label{lem:uniqueness_EL_equation}
There exists at most one solution $f$ of \eqref{eq:EL_equation_entropy} such that $f\in C^{k,\alpha}(M)$ and $f\to 0$ at infinity.
Moreover, for every bounded domain $\Omega$ with smooth boundary, there is at most one solution of \eqref{eq:EL_equation_entropy} such that 
$f\in C^{k,\alpha}(\Omega)\cap C^0(\overline{\Omega})$  
and $f|_{\partial \Omega} = 0$.
\end{lem}

\begin{proof}
We consider the case where $f\to 0$ at infinity. The other case is treated similarly.
Suppose $f_1,f_2$ are solutions of \eqref{eq:EL_equation_entropy} with $f_1,f_2\to0$ at infinity. The difference $f_0=f_1-f_2$ satisfies
\begin{align}\label{eq:EL_equation_entropy_2}
2\Delta f_0+\langle \nabla(f_1+f_2), \nabla f_0 \rangle + 2(n-1)f_0=0.
\end{align}
We have $f_0\to 0$ at infinity, so it either vanishes identically or it has a  maximum or a minimum in the interior.
If $f_0$ attains a maximum at an interior point $p$, we have $f_0(p)>0$, 
$\nabla f_0(p)=0$ and $\Delta f_0(p)\geq0$ which contradicts \eqref{eq:EL_equation_entropy_2}. The argument for the minimum is analogous. Therefore, $f_0=0$, which proves uniqueness.
\end{proof}
\noindent For a bounded domain $\Omega\subset M$, we define
\begin{align*}
\overline{\mathcal{W}}_{\Omega}(g,u) &=\int_{\Omega} 
\left( 4|\nabla u|^2+(\scal_g +n(n-1))u^2\right) \dv_g\\
&\qquad+2\int_{\Omega}\left( (\scal_g +n(n-1))u+G(u) \right) \dv_g.
\end{align*}
with associated localized entropy
\begin{align*}
\mu_{\Omega}(g)=\inf_{u\in C^{\infty}_c(\Omega)}\overline{\mathcal{W}}_{\Omega}(g,u).
\end{align*}
Note that
\begin{align}\label{eq:local_lagrangian}
\overline{\mathcal{W}}_{\Omega}(g,u)=
\overline{\mathcal{W}}_{\rm{AH},\hg}(g,u)- S_{\hg}(g)
\end{align}
for all $u\in C^{\infty}_c(\Omega)\subset C^{\infty}_c(M)$.
\begin{prop}\label{prop:bdry_value_problem}
Let $\Omega\subset M$ be a bounded domain with smooth boundary. Then there exists a function $u$ that realizes the infimum of $\overline{\mathcal{W}}_{\Omega}(g,\cdot)$. The function $f$ such that
\begin{align*}
e^{-f}=(u+1)^2
\end{align*} 
is then the unique solution of \eqref{eq:EL_equation_entropy} such that $f|_{\partial \Omega}= 0$.
\end{prop}
\begin{proof}
 From \eqref{eq:local_lagrangian} and Lemma \ref{lem_H1-comparison}, we get
\begin{align*}
C(\left\|u\right\|_{H^1}^2-1)
\leq \overline{\mathcal{W}}_{\Omega}(g,u) \leq 
C \left( (1+\left\|u\right\|_{H^1}^{\epsilon}) \left\|u\right\|_{H^1}^2+1\right).
\end{align*}
Thus,
\begin{align*}
\mu_{\Omega}(g)=\inf_{u\in H^1_0(\Omega)}\overline{\mathcal{W}}_{\Omega}(g,u)>-\infty.
\end{align*}
 Let $u_i$ be a minimizing sequence for $\mu_{\Omega}(g)$. The sequence is obviously bounded in $H^1$, hence there exists a subsequence, again denoted by $u_i$, which converges weakly in $H^1$ and strongly in $L^p$ for some fixed $p<\frac{2n}{n-2}$ to a function $u\in H^1_0(\Omega)$. The functional $u\mapsto \overline{\mathcal{W}}_{\Omega}(g,u)$ is lower semicontinuous in $H^1_0(\Omega)$. Therefore,
\begin{align*}
\mu_{\Omega}(g)\leq \overline{\mathcal{W}}_{\Omega}(g,u)&\leq\liminf_{i\to\infty}
\overline{\mathcal{W}}_{\Omega}(g,u_{i})\leq \mu_{\Omega}(g).
\end{align*}
Consequently, $u\in H^1_0(\Omega)$ is the desired minimizer. By variational calculus, the function $u$ is a weak solution of \eqref{eq:EL_equation_entropy}. Standard arguments involving elliptic regularity and Sobolev embedding yield $u\in C^{k,\alpha}$ and uniqueness holds due to Lemma \ref{lem:uniqueness_EL_equation}.
\end{proof}

\begin{lem} \label{lem:uniform_bounds}
Let $\Omega\subset M$ be a bounded domain with smooth boundary and  $f\in C^{k,\alpha}(\Omega)\cap C^0(\overline{\Omega})$ be a solution of \eqref{eq:EL_equation_entropy} with $f|_{\partial \Omega}= 0$.
Then
\begin{align*}
\frac{1}{2(n-1)}\inf_M(\scal_g +n(n-1)) 
\leq f \leq 
\frac{1}{2(n-1)}\sup_M(\scal_g +n(n-1)).
\end{align*}
\end{lem}

\begin{proof}
Let $x\in\Omega$ be the point where $f$ attains its maximum. It follows that $\nabla f(x)=0$ and $\Delta f(x)\geq0$. From \eqref{eq:EL_equation_entropy} we get
\begin{align*}
0 &=
2\Delta f(x) +  |\nabla f|^2-\scal(x)-n(n-1)+2(n-1)f(x) \\
&\geq 
2(n-1)f(x)-\scal(x)-n(n-1),
\end{align*}
which implies the upper bound. The argument for the lower bound is similar.
\end{proof}

We can now prove existence on the whole manifold.

\begin{thm} \label{thm:existence_on_M}
There are unique bounded functions $u_g,f_g\in C^{k,\alpha}(M)$ with $e^{-f_g}=(u_g+1)^2$, such that
\begin{align*}
\mu_{\rm{AH},\hg}(g)=\mathcal{W}_{\rm{AH},\hg}(g,f_g)=
\overline{\mathcal{W}}_{\rm{AH},\hg}(g,u_g).
\end{align*}
\end{thm}

In other words, there exists unique functions realizing the infimum in the definition of the entropy.
\begin{proof}
Let $\Omega_i$ be a sequence of bounded domains with smooth boundaries such that $\Omega_i\subset \Omega_{i+1}$ and $\cup_{i=1}^{\infty} \Omega_i=M$. Let the sequence $f_i$ be the solutions of \eqref{eq:EL_equation_entropy} such that $f_i\in C^{k,\alpha}(\Omega_i)\cap C^0(\overline{\Omega}_i)$ and $f_i|_{\partial \Omega_i} = 0$. By Lemma \ref{lem:uniform_bounds}, we have uniform bounds
\begin{align*}
\frac{1}{2(n-1)}\inf_M(\scal_g +n(n-1))
\leq f_i \leq 
\frac{1}{2(n-1)}\sup_M(\scal_g +n(n-1)).
\end{align*}
Thus, there exists a subsequence, which we also denote by $f_i$, which converges locally uniformly in $C^{k,\alpha}$ to a function $f_g$ defined on the whole manifold, which again solves \eqref{eq:EL_equation_entropy}. Note that $f$ necessarily satisfies the same bounds.

It remains to show that $f_g$ is the minimizer of ${\mathcal{W}}_{\rm{AH},\hg}(g,\cdot)$, or equivalently, the associated function $u_g$ is the minimizer of $\overline{\mathcal{W}}_{\rm{AH},\hg}(g,\cdot)$. By domain monotononicity we have
\begin{align*}
\mu_{\Omega_i}(g)\geq \mu_{\Omega_{i+1}}(g)
\end{align*}
and since $\cup_{i=1}^{\infty} \Omega_i=M$ we get from \eqref{eq:local_lagrangian} that
\begin{align*}
\lim_{i\to \infty}\mu_{\Omega_i}(g) + S_{\hg}(g) = \mu_{\rm{AH},\hg}(g).
\end{align*}
In particular, we have an upper bound on $\mu_{\Omega_i}(g)$.
Let $u_i$ be the minimizer of $\overline{\mathcal{W}}_{\Omega_i}(g,\cdot)$.
Then,
\begin{align*}
&\mu_{\rm{AH},\hg}(g) - S_{\hg}(g) \\
&\qquad=\lim_{i\to \infty}\mu_{\Omega_i}(g)
=\lim_{i\to \infty}{\mathcal{W}}_{\Omega}(g,f_i)
=\lim_{i\to \infty}\overline{\mathcal{W}}_{\Omega}(g,u_i)\\
&\qquad=
\lim_{i\to \infty}\int_M \left( 4|\nabla u_i|^2+(\scal_g +n(n-1))u_i^2
+2(\scal_g +n(n-1))u_i+G(u_i)\right) \dv_g.
\end{align*}
As in the proof of Lemma \ref{lem:uniform_bounds}, we can estimate
\begin{align*}
&\int_M 
\left| 4|\nabla u_i|^2+(\scal_g +n(n-1))u_i^2 + 2(\scal_g +n(n-1))u_i+G(u_i) \right|
\dv_g\\
&\qquad\leq
\int_M \left( 4|\nabla u_i|^2+|\scal_g +n(n-1)|u_i^2
+2|\scal_g +n(n-1)| \cdot |u_i|+|G(u_i)| \right)\dv_g\\
&\qquad\leq C \left( 1+(\left\|u_i\right\|_{H^1}^{2})^{\epsilon/2} \right) 
\left\|u_i\right\|_{H^1}^2+C\\
&\qquad\leq C \left(1 + C(\overline{\mathcal{W}}_{\Omega_i}(g,u_i)+1)^{\epsilon/2}\right)
(\overline{\mathcal{W}}_{\Omega_i}(g,u_i)+1)\\
&\qquad\leq 
C\left( 1+C(\mu_{\Omega_i}(g)+1)^{\epsilon/2}\right) (\mu_{\Omega_i}(g)+1) 
\leq C.
\end{align*}
Note that from the third to fourth line, we used the lower bound from Lemma~\ref{lem:uniform_bounds}. Note also that the  constants involved are independent of $\Omega_i$. Thus, since $u_i\to u_g$ locally uniformly in all derivatives, we can apply the dominated convergence theorem and conclude
\begin{align*}
\mu_{\rm{AH},\hg}(g)
&=
\lim_{i\to \infty}\int_M 
\big(4|\nabla u_i|^2+(\scal_g +n(n-1))u_i^2\\
&\qquad+2(\scal_g +n(n-1))u_i+G(u_i) \big)
\dv_g+ S_{\hg}(g)\\
&=
\int_M 
\big(4|\nabla u_g|^2+(\scal_g +n(n-1))u_g^2 \\
&\qquad +2(\scal_g +n(n-1))u_g+G(u_g) \big) \dv_g
+ S_{\hg}(g)\\
&=\overline{\mathcal{W}}_{\rm{AH},\hg}(g,u_g)\\
&=\mathcal{W}_{\rm{AH},\hg}(g,f_g),
\end{align*}
which is the desired result.
\end{proof}

Next, we consider the asymptotics of the minimizing function $f_g$.
\begin{lem}\label{lem:super_sub_solns}
For each constant $C>0$ there exist a positive function $f_+$
satisfying
\begin{align}\label{eq:super_solns}
2\Delta f_++|\nabla f_+|^2-\scal_g-n(n-1)+2(n-1)f_+\geq 0
\end{align}
 and a negative function $f_-$ such that
\begin{align}\label{eq:sub_solns}
2\Delta f_-+|\nabla f_-|^2-\scal_g-n(n-1)+2(n-1)f_-\leq 0
\end{align}
both defined on the complement $M\setminus B_R$ of a large ball $B_R$, such that $\pm f_{\pm}|_{\partial B_R}>  C$ and $f_{\pm}=O(e^{-2\mu r})$ for some $\mu>0$.
\end{lem}
\begin{proof}[Proof of Lemma \ref{lem:super_sub_solns}]
Given a metric $\sigma_0$ on the conformal boundary of $(M,g$), we can choose a boundary defining function $\rho$ such that on a neighborhood $U\subset M$ near $\partial N$, $g$ is of the form
\begin{align*}
g=\rho^{-2}(d\rho^2+\sigma_\rho),
\end{align*}
where $\rho\mapsto \sigma_{\rho}$ is a $C^{k,\alpha}$ family of Riemannian metrics on $\partial N$, see \cite[Theorem 2.11]{GraLee91}.
With $\rho=e^{r}$, we obtain
\begin{align*}
g=dr^2+e^{2r}\sigma_{r}.
\end{align*}
Note that $\partial_rh=-e^ {-r}\partial_{\rho}\sigma$, so that $|\partial_r\sigma|_\sigma=O(e^{-r})$ as $r\to\infty$. The Laplacian of $g$ is
\begin{align*}
\Delta_gf=-\partial^2_{rr}f-(n-1)\partial_rf-\frac{1}{2}(\trace_{\sigma}\partial_r\sigma)\partial_r f+e^{-2r}\Delta_{\sigma_r}f.
\end{align*}
Suppose, without loss of generality, that $U=\left\{x\in M\mid r(x)>R_0\right\}$ for some radius $R_0$ and let $f:U\to \R$ be a function depending only on $r$. Define functions on $M\setminus B_{R_0}$ by $a=\trace_\sigma(\partial_r \sigma)$ and $b=\scal_g+n(n-1)$.
Then, $f$ satisfies \eqref{eq:super_solns} (resp.\ \eqref{eq:sub_solns}) if and only if
\begin{align*}
-2\partial^2_{rr}f-2(n-1)\partial_rf-a \partial_rf+(\partial_rf)^2-b+2(n-1)f\geq 0 \text{ }(\text{resp.}\leq 0).
\end{align*}
With the ansatz $f_+(r)=\lambda e^{-\mu r}$, $f_+$ satisfies \eqref{eq:super_solns} if
\begin{align*}
\lambda \left(-2\mu^2+2\mu(n-1)+2(n-1)\right)e^{-\mu r}\geq b-\lambda\mu ae^{-\mu r}+(\lambda \mu)^2e^{-2\mu r}.
\end{align*}
We know that there exist constants $A,B\geq0$ such that $|a|\leq A e^{-r}$ and $|b|\leq B e^{-\delta r}$. Choose $R\geq R_0$ and $\lambda>\max\left\{C,\left(2(n-1)\right)^{-1}Be^{-\delta R} \right\}$, where $C>0$ is the constant in the statement of the lemma. Now choose $\mu\in (0,\delta]$ so small that
\begin{align*}
\lambda e^{-2\mu R}&> C,\\
\lambda \left(-2\mu^2+2\mu(n-1)+2(n-1)\right)e^{-\mu R}&\geq Be^{-\delta R}+A\lambda\mu e^{-(\mu+1)R}
 +\mu^2\lambda^2e^{-2\mu R}.
\end{align*}
Because $\mu\leq \delta$, we get for all $r\geq R$ that 
\begin{align*}
\lambda \left(-2\mu^2+2\mu(n-1)+2(n-1)\right)e^{-\mu r}&\geq B e^{-\delta r}+A\lambda\mu e^{-(\mu+1) r}+(\lambda\mu)^2e^{-2\mu r}\\
&\geq b-\lambda \mu ae^{-\mu r}+(\lambda\mu)^2e^{-2\mu r}
\end{align*}
for all $r\geq R$. Thus for these choices of $\mu$ and $\lambda$, $f_+$  is a positive function satisfying \eqref{eq:super_solns} on $M\setminus B_{R}$ with $f_+(R)> C$.

Similarly, with the ansatz $f_-(r)=-\lambda e^{-\mu r}$, the function $f_-$ satisfies \eqref{eq:sub_solns} if
\begin{align*}
-\lambda \left(-2\mu^2+2\mu(n-1)+2(n-1)\right)e^{-\mu r}\leq b+A\lambda\mu e^{-\mu r}+(\lambda \mu)^2e^{-2\mu r}.
\end{align*}
Let $\lambda, \mu$ and $R$ be as before. Then,
\begin{align*}
-\lambda  \left(-2\mu^2+2\mu(n-1)+2(n-1)\right)e^{-\mu r}&\leq -B  e^{-\delta r}-AC\mu e^{-(\mu+1) r}\\
&\leq b-\lambda\mu ae^{-\mu r}+(\lambda \mu)^2e^{-2\mu r}
\end{align*}
for all $r\geq R$ and $f_-$  is a negative function satisfying \eqref{eq:sub_solns} on $M\setminus B_{R}$ with $f_-(R)< -C$.
\end{proof}

\begin{lem}\label{lem:minimizer_little_decay}
Let $f$ be the minimizer given by Theorem~\ref{thm:existence_on_M}. Then, $f\in O(e^{-2\mu r})$ for some $\mu>0$.
\end{lem}
\begin{proof}
Let $\Omega_i$ be a sequence of bounded domains such as in the proof of  Theorem \ref{thm:existence_on_M}. Let $f_i$ be the solutions of the  Dirichlet problem for \eqref{eq:EL_equation_entropy} on $\Omega_i$, which exists by Proposition \ref{prop:bdry_value_problem}. By Lemma \ref{lem:uniform_bounds}, $|f_i|\leq C$ for some constant $C>0$.

Let $f_+$ and $f_-$ be the functions provided by Lemma \ref{lem:super_sub_solns}.
We are done with the proof if we are able to show $f_-\leq f\leq f_+$. Since a subsequence of the $f_i$ converges locally uniformly to $f$, it suffices to show $f_-\leq f_i\leq f_+$. We may assume that $i$ is so large that $\Omega_i\setminus B_R$ is an annular region with inner boundary $\partial B_R$  and outer boundary  $ \partial \Omega_i$.
On this set,  the functions $f_{+},f_{-}$ and $f_i$ are all defined. 
%
%
%
We consider the inequality $f_i\leq f_+$, the other one is similar. 

Assume that there exists an $i$ for which the inequality $f_i\leq f_+$ fails to hold. We know that $f_i=0<f_+$ on $\partial \Omega_i$ and $f_i\leq C<f_+$ on $\partial B_R$, so that the inequality fails in the interior. Certainly, we have $f_++C>C>f_i$ on $\Omega_i\setminus B_R$. Therefore, there exists a smallest $C_0>0$ for which the inequality $f_{+}+C_0\geq f_i$ holds, and a point $x$ in the interior of the compact set $ \Omega_i\setminus B_R$ such that $f_+(x)+C_0=f_i(x)$. At this point, we have $\nabla f_+(x)=\nabla f_i(x)$ and $\Delta f_i(x)\geq \Delta f_+(x)$. However, since $f$ solves \eqref{eq:EL_equation_entropy} and $f_+$ satisfies \eqref{eq:super_solns}, we get
\begin{align*}
0&=2\Delta f_i(x)+|\nabla f_i(x)|^2-\scal_g(x)-(n-1)n+2(n-1)f_i(x)\\
&\geq 2\Delta f_+(x)+|\nabla f_+(x)|^2-\scal_g(x)-(n-1)n+2(n-1)(f_+(x)+C_0)\\
&\geq 2(n-1)C_0,
\end{align*}
which contradicts $C_0>0$.
\end{proof}
Recall that under the assumptions we made at the beginning of this section, $(M,g)$ is asymptotic to $(\widehat M,\hg)$ of order $\delta>\frac{n-1}{2}$. For technical reasons, we will additionally assume for the remainder of this section that  $\delta<\frac{n-1}{2}+\frac{1}{2}\sqrt{(n+3)(n-1)}$.
\begin{lem}\label{lem:entropyanalytic}
Let the function $f$ be the minimizer given by Theorem~\ref{thm:existence_on_M}.
Then $f\in C^{k,\alpha}_{\delta}(M)$.
\end{lem}
\begin{proof}
By substituting $u+1=e^{-f/2}$, equation \eqref{eq:EL_equation_entropy} transforms to 
\begin{align}\label{eq_EL_equation_u}
4\Delta u=-(\scal_g +(n-1)n)u-\scal_g-(n-1)n-H(u),
\end{align}
where
\begin{align*}
H(u)=2(n-1)\log((u+1)^2)(u+1).
\end{align*}
Note that for every $\epsilon>0$, there is a constant $C_{\epsilon}$ such that
\begin{align}\label{eq:estimate_nonlinearity}
\left| H(u) \right|
\leq C_{\epsilon}(|u| + |u|^{1+\epsilon}). 
\end{align}
By Lemma \ref{lem:minimizer_little_decay}, $f=O(e^{-\mu r})$ for some $\mu\in (0,\delta)$. Thus, we also get $u=O(e^{-\mu r})$, so that $u\in C^0_{\mu}$ and hence also $H(u)\in C^0_{\mu}$.
 Choose $\eta\in (0,\mu)$ and a sufficiently large $p\in (n,\infty)$ such that $\eta+\frac{n-1}{p}<\mu$. Then $u\in L^p_{\eta}$. Since $\scal_g +(n-1)n\in O (e^{-\delta r})$ with $\delta>\eta$ and $H(u)=O(e^{-\mu r})$, we have
 \begin{align*}
 (\scal_g +(n-1)n)u-\scal_g-(n-1)n-H(u)\in L^p_{\eta},
 \end{align*}
and $u\in W^{2,p}_{\eta}\subset C^{1,\alpha}_{\eta}$ by elliptic regularity applied to \eqref{eq_EL_equation_u} and Sobolev embedding (see \cite[Lemma 3.6]{Lee06}. Hence, we also have $f\in C^{1,\alpha}_{\eta}$. Now we are going to to successively improve the decay of $f$.
Since $f$ satisfies \eqref{eq:EL_equation_entropy}, we get
\begin{align}
2\Delta f+2(n-1)f=-|\nabla f|^2+\scal_g+n(n-1).
\end{align}
We have that $|\nabla f|^2\in C^{0,\alpha}_{2\eta}$ and $\scal_g+n(n-1)\in C^{k-2,\alpha}_{\delta}$. By \cite[Propositions C and E]{Lee06}, the operator $\Delta +(n-1):C^{l,\alpha}_{\beta}\to C^{l-2,\alpha}_{\beta}$ is an isomorphism for all $2\leq l\leq k$ and 
\begin{align}\label{eq:weights}
 \frac{n-1}{2}-\frac{1}{2}\sqrt{(n+3)(n-1)}<   \beta< \frac{n-1}{2}+\frac{1}{2}\sqrt{(n+3)(n-1)} .
\end{align}
Thus we get $f\in C^{k,\alpha}_{\eta_1}$ for $0<\eta_1\leq\min \left\{2\eta,\delta\right\}$, since we assume $\delta<\frac{n-1}{2}+\frac{1}{2}\sqrt{(n+3)(n-1)} $. Using the above arguments again, we obtain $f\in C^{k,\alpha}_{\eta_2}$ for $0<\eta_2\leq\min \left\{2\eta_1,\delta\right\}$.
After repeating this procedure  a finite number of times, we obtain the desired result.
\end{proof}

We now prove that the renormalized expander entropy depends analytically on the Riemannian metric.

\begin{prop} \label{lem_analyticity}
The map 
\begin{align*}
\mathcal{R}^{k,\alpha}_{\delta}(M) \ni
 g\mapsto f_g \in C^{k-2,\alpha}_{\delta}(M),
\end{align*} 
 associating to a metric $g$ the unique minimizer in the definition of $\mu_{\rm{AH},\hg}(g)$ is analytic. In particular, $\mathcal{R}^{k,\alpha}_{\delta}(M) \ni g\mapsto \mu_{\rm{AH},\hg}(g)$ is analytic.
\end{prop}
\begin{proof}
We consider the map $\Phi: \mathcal{R}^{k,\alpha}_{\delta}(M)\times C^{k,\alpha}_{\delta}(M) \to C^{k-2,\alpha}_{\delta}(M)$, given by 
\begin{align*}
\Phi(g,f) = 2\Delta f+|\nabla f|^2-\scal_g-n(n-1)+2(n-1)f.
\end{align*}
The differential of $\Phi$ in the second argument is given by
\begin{align}\label{eq_operator_P}
D_{(g,f)}\Phi(0,v)=2\Delta v+2\langle\nabla f,\nabla v\rangle+2(n-1)v.
\end{align}
The result will follow from the implicit function theorem, if we show that the map
\begin{align*}
P_{g,f} = D_{(g,f)}\Phi(0,.):  C^{k,\alpha}_{\delta}(M) \to C^{k-2,\alpha}_{\delta}(M)
\end{align*}
is an isomorphism for each $f\in C^{k,\alpha}_{\delta}$.

 In fact, an integration by parts argument with respect to the weighted measure $e^{-f}\dv$ shows that $P_{g,f}$ has trivial kernel. 
It remains to show that it is surjective. For this purpose,
one computes that 
$P_{g,f} = e^{f/2}\circ Q_{g,f}\circ e^{-f/2}$, 
where
\begin{align*}
Q_{g,f}(v) = 2\Delta v +2(n-1)v+\left(\Delta f+\frac{1}{2}|\nabla f|^2\right)v.
\end{align*}
By \cite[Proposition F]{Lee06} and the assumptions on $\delta$,
 \begin{align*}2\Delta +2(n-1): C^{k,\alpha}_{\delta}(M) \to C^{k-2,\alpha}_{\delta}(M)
 \end{align*}
  is Fredholm of index zero. Since  $\Delta f+\frac{1}{2}|\nabla f|^2\in O(e^{-\delta r})$ as $r\to\infty$, the operator $Q_{f}$ has the same indicial root as $2\Delta +2(n-1)$. Thus by \cite[Theorem C]{Lee06},
 $Q_{g,f}: C^{k,\alpha}_{\delta}(M) \to C^{k-2,\alpha}_{\delta}(M)$ is also a Fredholm operator of index zero. Because multiplication with $e^{\pm f}$ is an isomorphism, $P_{g,f}$ is also Fredholm of index zero. Since its kernel is zero, $P_{g,f}$ is an isomorphism, as desired.
\end{proof}

\subsection{First and second variation}
In the next proposition, we compute the first variation of $g\mapsto \mu_{\rm{AH},\hg}$.
\begin{prop}\label{prop:first_variation_entropy}
We have
\begin{align*}
D_g\mu_{\rm{AH},\hg}[h]=-\int_M \langle \ric+\nabla^2f_g+(n-1)g,h\rangle e^{-f_g}\dv ,
\end{align*}
where $f_g$ is the minimizing metric in the definition of $\mu_{\rm{AH},\hg}(g)$.
\end{prop}
\begin{proof}
If $h\in C^{k,\alpha}_{\delta}$, then by Proposition \ref{lem_analyticity}, $v=\frac{d}{dt}f_{g+th}|_{t=0}\in C^{k,\alpha}_{\delta}(M)$. In order to avoid cumbersome notation, let us write $f=f_g$ throughout the rest of the proof.
An approximation argument shows that the computation in the proof of 
Proposition \ref{prop:EL_equation_entropy}
is also valid for $v\in C^{k,\alpha}_{\delta}$. Thus, since $f$ is the minimizer,
\begin{align*}
D_{g,f}\mathcal{W}_{\rm{AH},\hg}[0,v]=0,
\end{align*}
and the chain rule implies
\begin{align*}
D_g\mu_{\rm{AH},\hg}[h] 
&= D_{g,f}\mathcal{W}_{\rm{AH},\hg}[h,v] \\
&= D_{g,f}\mathcal{W}_{\rm{AH},\hg}[h,0]+D_{g,f}\mathcal{W}_{\rm{AH},\hg}[0,v]\\
&= D_{g,f}\mathcal{W}_{\rm{AH},\hg}[h,0].
\end{align*}
Using this and the first variation of the scalar curvature, we compute
\begin{align*}
&D_g\mu_{\rm{AH},\hg}[h]
=\frac{d}{dt} \mathcal{W}_{\rm{AH},\hg}(g+th,f) |_{t=0} \\
&\qquad=\frac{d}{dt} \int_M 
\left( 
\left(|\nabla f|^2+\scal +n(n-1) \right)e^{-f} - 2(n-1)\left((f+1)e^{-f} - 1 \right) \right) \dv |_{t=0} \\
&\qquad\qquad
-\frac{d}{dt} \m_{\rm{VR},\hg}(g) |_{t=0}\\
&\qquad=\int_M 
\left( -\langle h,\nabla f\otimes\nabla f\rangle
+ \Delta\trace h + \mathrm{div}(\mathrm{div} h) - \langle\ric,h\rangle \right) e^{-f}\dv_g\\
&\qquad\qquad +\frac{1}{2}\int_M 
\left( \left( |\nabla f|^2+\scal +n(n-1) \right) e^{-f} - 2(n-1)\left((f+1)e^{-f}\right) \right) \trace h \dv\\
&\qquad\qquad-
\lim_{\eta\to\infty}\int_{\partial B_R}\langle \mathrm{div} h-\nabla\trace h,\nu\rangle \dv.
\end{align*}
By integration by parts over a large ball, we have
\begin{align*}
 \int_{B_R}\Delta(\trace h)e^{-f}\dv&= \int_{B_R}\trace h\Delta(e^{-f})\dv - \int_{\partial B_R}\langle \nabla \trace h,\nu\rangle e^{-f}\dv\\
&\qquad- \int_{\partial B_R}\trace h\langle \nabla f,\nu\rangle e^{-f}\dv\\
&=-\int_{B_R}\trace h(\Delta f+|\nabla f|^2) e^{-f}\dv - \int_{\partial B_R}\langle \nabla \trace h,\nu\rangle e^{-f}\dv\\
&\qquad- \int_{\partial B_R}\trace h\langle \nabla f,\nu\rangle e^{-f}\dv
\end{align*}
and
\begin{align*}
\int_{B_R} \mathrm{div}(\mathrm{div} h)e^{-f}\dv&=\int_{B_R}\langle h,\nabla^2(e^{-f})\rangle\dv\\
&\qquad+ \int_{\partial B_R}\langle \mathrm{div} h+h(\nabla f,\cdot),\nu\rangle e^{-f}\dv\\
&=\int_{B_R}\langle h,\nabla f\otimes\nabla f-\nabla^2f\rangle e^{-f}\dv\\
&\qquad+\int_{\partial B_R}\langle\mathrm{div} h+h(\nabla f,\cdot),\nu\rangle e^{-f}\dv.
\end{align*}
Since $h\in C^{k,\alpha}_{\delta}$ and $f\in C^{k,\alpha}_{\delta}$ for some $\delta>\frac{n-1}{2}$, we have 
\begin{align*}
\left| \int_{\partial B_R}\trace h\langle \nabla f,\nu\rangle e^{-f}\dv \right|
+
\left| \int_{\partial B_R}\langle h(\nabla f,\cdot),\nu\rangle e^{-f}\dv \right|
\to 0
\end{align*}
as well as
\begin{align*}
\left| \int_{\partial B_R}\langle \nabla \trace h,\nu\rangle(1- e^{-f})\dv \right|
+
\left| \int_{\partial B_R} \langle \mathrm{div} h,\nu\rangle (1 - e^{-f}) \dv \right| 
\to 0
\end{align*}
as $R\to \infty$.
Therefore, all the boundary terms vanish and after summing up, we obtain
\begin{align*}
D_g\mu_{\rm{AH},\hg}[h]
&=
-\int_{M}\trace h(\Delta f+|\nabla f|^2) e^{-f}\dv-\int_M \langle \nabla^2f+\ric,h\rangle e^{-f}\dv\\
&\qquad
+\frac{1}{2}\int_M \left([|\nabla f|^2+\scal +n(n-1)]e^{-f}
-2(n-1)(f+1)e^{-f}\right)\trace h\dv\\
&=
-\int_M \langle \nabla^2f+\ric,h\rangle e^{-f}\dv
-(n-1)\int_M\trace h e^{-f}\dv \\
&\qquad
+\frac{1}{2}\int_M [-2\Delta f-|\nabla f|^2+\scal +n(n-1)-2(n-1)f]e^{-f}\trace h\dv.
\end{align*}
Using the Euler--Lagrange equation \eqref{eq:EL_equation_entropy}, the last integral on the right hand side vanishes and we obtain the first variation formula as stated.
\end{proof}

It follows that critical points of the renormalized expander entropy are precisely the Einstein metrics. In what follows we will continue to use $f$ to denote the minimizing function in the definition of $\mu_{\rm{AH},\hg}(g)$.

\begin{cor}\label{cor:Entropy_critical_points}
A metric $g\in\mathcal{R}_{\delta}^{k,\alpha}$ is a critical point of $\mu_{\rm{AH},\hg}$ if and only if it is PE.
\end{cor}

\begin{proof}
By Proposition \ref{prop:first_variation_entropy},
the critical points of $\mu_{\rm{AH},\hg}$ satisfy
\begin{align*}
\ric+\nabla^2 f=-(n-1)g,
\end{align*}
or equivalently
\begin{align*}
\ric-\frac{1}{2}\scal \, g=-\nabla^2f-\frac{1}{2}(\Delta f) g-(n-1)g+\frac{n}{2}(n-1)g.
\end{align*}
By the second Bianchi identity, the left-hand side is divergence free, so that
\begin{align*}
0&=\mathrm{div}\nabla^2f+\frac{1}{2}\mathrm{div}((\Delta f) g)=
-\Delta \nabla f+\frac{1}{2}\nabla(\Delta f)\\
&=-\Delta \nabla f+\frac{1}{2}(\Delta+\ric)(\nabla f)
=-(\Delta+\nabla^2f+\frac{n-1}{2})(\nabla f) \\
&=-(\Delta+\nabla_{\nabla f}+\frac{n-1}{2})(\nabla f).
\end{align*}
Since $\nabla f\in H^1$, taking the scalar product of the equation with  $\nabla f$ itself and integrating over $M$ with respect to the weighted measure $e^{-f}\dv$ yields $\nabla f=0$. Thus, we have $\nabla^2f=0$ which implies that $\ric=-(n-1)g$, as desired.
\end{proof}

\begin{lem}\label{lem_diff_invariance_entropy} We have
\begin{align*}
D_g\mu_{\rm{AH},\hg}[\mathcal{L}_Xg]=0
\end{align*}
for any $X\in C^{k-1,\alpha}_{\delta}(TM)$.
\end{lem}

\begin{proof}
Let $f=f_g$ be the minimizer in the definition of $\mu_{\rm{AH},\hg}$ and let $u=e^{-f/2}-1$. By the proof of Lemma \ref{lem_W_entropy},
\begin{equation}\begin{split}\label{eq:entropy_in_u}
\mu_{\rm{AH},\hg}(g)&=\overline{\mathcal{W}}_{\rm{AH},\hg}(g,u)\\ &=
\int_M \Big( 4|\nabla u|^2+(\scal_g +n(n-1))u^2\\
&\qquad +2(\scal_g +n(n-1))u+F(u) \Big) \dv_g 
+ S_{\hg}(g),
\end{split}
\end{equation}
where $F(x)=2(n-1) \left( (\log((x+1)^2)-1)(x+1)^2+1 \right)$.

By density of $ C^{k+1,\alpha}_{\delta}$ in $ C^{k-1,\alpha}_{\delta}$, it suffices to show the lemma for  $X\in C^{k+1,\alpha}_{\delta}(TM)$. Let $\varphi_t$ be the diffeomorphisms generated by $X$ and $g_t=\varphi_t^*g$. By diffeomorphism invariance of the Euler--Lagrange equation \eqref{eq:EL_equation_entropy}, $f_{g_t}=f_g\circ \varphi_t$, so that also $u_{g_t}=u_g\circ \varphi_t$. Furthermore, Corollary \ref{cor:diff_inv:EH} 
 implies that $S_{\hg}(g_t)$ is constant along $t$. Using this together with $u_{g_t}=u_g\circ \varphi_t$ in \eqref{eq:entropy_in_u} implies that also $\mu_{\rm{AH},\hg}(g)$ is constant in $t$. Differentiating at $t=0$ gives the desired result.
\end{proof}

We have mentioned Ricci flow on AH manifolds but have yet to discuss it in detail. In the AH setting, it natural to normalize the Ricci flow equation to
\begin{align} \label{eq:AH_Ricci_flow}
\frac{d}{dt}g_t=-2\ric_{g_t}+2(n-1)g_t.
\end{align}
Existence and uniqueness of solutions to \eqref{eq:AH_Ricci_flow} in the AH setting has been established by Bahuaud \cite{Bah11}.
While Bahuaud only considers AH metrics with smooth compactifications,  these results can probably be extended to AH metrics of class $C^{k,\alpha}_{\delta}$ along the same lines. Thus if $g_0\in \mathcal{R}^{k,\alpha}_{\delta}(M,\hg)$, we expect to have a unique solution $g_t$ of \eqref{eq:AH_Ricci_flow} starting at $g_0$ such that $g_t\in \mathcal{R}^{k,\alpha}_{\delta}(M,\hg)$.

\begin{lem}
Let $g_t$ be a solution of \eqref{eq:AH_Ricci_flow} with $g_t\in\mathcal{R}^{k,\alpha}_{\delta}(M,\hg)$. Then, the function $t\mapsto \mu_{\rm{AH},\hg}(g_t)$ is monotonically increasing. It is strictly monotonically increasing unless $g_t$ is a constant family of PE metrics.
\end{lem}

\begin{proof}
By Lemma \ref{lem_diff_invariance_entropy}, $D_g\mu_{\rm{AH},\hg}[\nabla^2 f]=0$ for $f\in C^{k,\alpha}_{\delta}$. If $g_t$ is a solution of \eqref{eq:AH_Ricci_flow}, we  get from Proposition \ref{prop:first_variation_entropy} that
\begin{align*}
\frac{d}{dt}\mu_{\rm{AH},\hg}(g_t)=2\int_M | \ric_{g_t}+\nabla^2f_{g_t}+(n-1)g_t|_{g_t}^2 e^{-f_{g_t}}\dv_{g_t}\geq0.
\end{align*}
 The equality case follows from Corollary \ref{cor:Entropy_critical_points}.
\end{proof}

Next, we compute the second variation of the expander entropy at an Einstein metric.

\begin{defn}
The \emph{Lichnerowicz Laplacian} $\Delta_L$ and the \emph{Einstein operator} $\Delta_E$, which act on symmetric tensor fields $h$, are defined by
\begin{align*}
\Delta_L h_{ij}&=\Delta h_{ij}+\ric_{ik}h_j^k+\ric_{jk}h_i^k-2h^{kl}R_{iklj},
\qquad\text{ and}\\
\Delta_Eh_{ij}&=\Delta h_{ij}-2h^{kl}R_{iklj}.
\end{align*}
Note that if the underlying manifold is Einstein with $\ric=\lambda g$ then we simply have $\Delta_Lh=(\Delta_E+2\lambda)h$.
\end{defn}
\begin{prop}\label{thm:second_variation_entropy} 
Let $(M,g)$ be a complete PE manifold. Then the second variation of $\mu_{\rm{AH},\hg}$ at $g$ is given by
\begin{align*}
D^2_g\mu_{\rm{AH},\hg}[h,h]=\begin{cases}
-\frac{1}{2}\int_M\langle \Delta_E h,h\rangle\dv, & \text{ if }\mathrm{div} h=0,\\
0, & \text{ if }h=\mathcal{L}_Xg \text{ for }X\in C^{k+1,\alpha}_{\delta}(TM).
\end{cases}
\end{align*}
Moreover, $D^2_g\mu_{\rm{AH},\hg}$ is diagonal with respect to the orthogonal decomposition
\begin{align}\label{eq_decomposition}
C^{k,\alpha}_{\delta}(S^2T^*M)=
\mathrm{ker}(\mathrm{div})
\oplus \left\{\mathcal{L}_Xg\mid X\in C^{k+1,\alpha}_{\delta}(TM)\right\}.
\end{align}
\end{prop}

\begin{proof}
Recall that the first variation of the Ricci tensor \cite[Theorem 1.174]{Bes08} is given by
\begin{equation*}
D_g\ric[h]=\frac12\left( \Delta_L h+\mathcal{L}_{\mathrm{div}h}g-\nabla^2\trace h \right).
\end{equation*} 
From this we	 compute for $h\in C^{k,\alpha}_{\delta}(S^2T^*M)$ with $\mathrm{div} h=0$,
\begin{align*}
&\frac{d^2}{dt^2}\mu_{\rm{AH},\hg}(g+th)|_{t=0}=-\frac{d}{dt}
\int_M\langle \ric+\nabla^2f)+(n-1)g,h\rangle e^{-f}\dv\big|_{t=0}\\
&\qquad=-\int_M\langle D_g\ric(h)+\nabla^2f'+(n-1)h,h\rangle  e^{-f}\dv\\
&\qquad=-\int_M\langle \frac{1}{2}[\Delta_Lh+\mathcal{L}_{\mathrm{div}h}g-\nabla^2\trace h]+\nabla^2f'+(n-1)h,h\rangle  e^{-f}\dv\\ 
&\qquad=-\frac{1}{2}\int_M\langle \Delta_Lh+\nabla^2(2f'-\trace h)+2(n-1)h,h\rangle,
\end{align*}
where $f'=\frac{d}{dt}f_{g+th}|_{t=0}$, and here we also used that $f= 0$. In order to compute  $f'$, we differentiate the Euler--Lagrange equation \eqref{eq:EL_equation_entropy}.
Since $f$ is constant and $\mathrm{div} h=0$, this yields
\begin{align*}
0&=2\Delta f'-D_g\scal(h)+2(n-1)f'\\
&=2(\Delta+(n-1))(f')-[\Delta \trace h+\delta(\delta h)-\langle\ric,h\rangle]\\
&=2(\Delta+n-1)(f')-(\Delta+n-1) \trace h.
\end{align*}
Recall that the operator $\Delta +n-1: C^{k,\alpha}_{\delta}(M)\to C^{k-2,\alpha}_{\delta}(M)$ is an isomorphism. This implies $2f'=\trace h$ and 
\begin{align*}
D^2_g\mu_{\rm{AH},\hg}[h,h]=-\frac{1}{2}\int_M\langle \Delta_E h,h\rangle\dv,
\end{align*}
which proves the variational formula for the case $\mathrm{div} h=0$.
For the other case, let $h=\mathcal{L}_Xg$ for some $X\in C^{k+1,\alpha}_{\delta}(TM)$ and $k\in C^{k,\alpha}_{\delta}(S^2T^*M)$ be arbitrary. Let $\varphi_t$ be the diffeomorphism generated by $X$ and consider the $2$-parameter family of metrics $g_{t,s}=\varphi_t^*(g+s k)$. Then by Lemma \ref{lem_diff_invariance_entropy}, $\mu_{\rm{AH},\hg}(g_{t,s})$ only depends on $s$, so that
\begin{align*}
D^2_g\mu_{\rm{AH},\hg}[h,k]=D^2_g\mu_{\rm{AH},\hg}[k,h]=\frac{d^2}{dsdt}\mu_{\rm{AH},\hg}(g_{t,s})\vert_{s,t=0}=0.
\end{align*}
Together with the fact that $\Delta_{E,\widetilde{g}}$ preserves the splitting \eqref{eq_decomposition} (see \cite[pp.\ 28--29]{Lic61}), the second variational formula as well as the orthogonality statement follow.
\end{proof}

\section{A local positive mass theorem}\label{sec:localPMT}

From this point onward, we will take the metric $\hat{g}$ to be a complete PE metric on the manifold $M=\widehat{M}$. We consider the functional $g\mapsto \mu_{\rm{AH},\hg}(g)$ on the space $\mathcal{R}^{k,\alpha}_{\delta}(M,\hg)$, with $\delta>\frac{n-1}{2}$.

\subsection{Local maxima of the entropy}
In order to establish a criterion for local maximality of $\mu_{\rm{AH},\hg}$ at $\hg$, we need to take into account that our second variation has an infinite-dimensional kernel. However due to diffeomorphism invariance, the space of Lie derivatives of the metric is
a subspace of finite codimension of the kernel.
Due to this fact, we may restrict our attention to the  subset
\begin{align*}
\mathcal{S}_{\hg}=\left\{g\in \mathcal{R}^{k,\alpha}_{\delta}(M,\hg)\mid \mathrm{div}_{\hg}g=0\right\}.
\end{align*}
Let $\mathrm{Diff}^{k+1,\alpha}_{\delta}(M)$ be the set of diffeomorphisms generated by the vector fields $X\in C^{k+1,\alpha}_{\delta}(TM)$.
\begin{lem}\label{lem:slice}
There exists a $C^{k,\alpha}_{\delta}$-neighbourhood $\mathcal{U}$ of $\hg$ in the space of metrics such that any $g\in\mathcal{U}$ can be uniquely written as $g=\varphi^*\widetilde{g}$ for some $\widetilde{g}\in \mathcal{U}\cap\mathcal{S}_{\hg}$ and a diffeomorphism $\varphi\in \mathrm{Diff}^{k+1,\alpha}_{\delta}(M)$ that is $C^{k+1,\alpha}_{\delta}$-close to the identity.
\end{lem}

\begin{proof}
Consider the smooth map
\begin{align*}
\Phi:\mathcal{S}_{\hg}\times \mathrm{Diff}^{k+1,\alpha}_{\delta}(M)&\to C^{k,\alpha}_{\delta}(S^2_+M),\\
(g,\varphi)&\mapsto \varphi^*g.
\end{align*}
Its differential at $(\hg,\identity)$ exactly corresponds to the decomposition \eqref{eq_decomposition}.
Therefore the assertion is an immediate consequence of the inverse function theorem.
\end{proof}

\begin{defn} \label{defn_lin_stab_int}
The PE manifold $(M,\widehat{g})$ is called \emph{linearly stable} if $\Delta_E$ is a nonnegative operator. It is called \emph{integrable} if there exists a $C^{k,\alpha}_{\delta}$-neighbourhood $\mathcal{U}$ of $\widehat{g}$ in the space of metrics such that
\begin{align*}
\mathcal{E}=\left\{g\in \mathcal{U}\cap\mathcal{S}_{\widehat{g}}\mid \ric_g=-(n-1)g\right\}
\end{align*}
is a smooth manifold with 
$T_{\hg}\mathcal{E}=\mathrm{ker}_{L^2}(\Delta_E)$.
\end{defn}

In order to prove a statement of local maximality for the entropy we also need to control some error terms.

\begin{lem} \label{lem_third_variation}
There exists a $H^k$-neighbourhood of $\widehat{g}$ in the space of metrics and a constant $C>0$ such that
\begin{align*}
\left| \frac{d^3}{dt^3} \mu_{\rm{AH},\hg}(g+th)\big|_{t=0}  \right|\leq C\left\| h\right\|_{C^{k,\alpha}_{\delta}}\left\| h\right\|_{H^1}^2
\end{align*}
for all $g\in\mathcal{U}$.
\end{lem}
\begin{proof}
We follow the proof of \cite[Proposition 4.5]{Kro20} for the expander entropy in the compact setting, to which we refer for further details.
Large part of the estimate follows from standard computations. The nontrivial part is to establish estimates on first and second variations of the minimizing function $f_g$. 
By differentiating \eqref{eq:EL_equation_entropy} twice, one sees that the defining equations for $v=\frac{d}{dt}f_{g+th}|_{t=0}$ and $w=\frac{d^2}{dt^2}f_{g+th}|_{t=0}$ are of the form
\begin{align*}
P_{g,f_g}(v)&=(*)\qquad P_{g,f_g}(w)=(**)
\end{align*}
for some right hand sides $(*),(**)$. Here, $P_{g,f_g}$ is the operator defined in \eqref{eq_operator_P}. In the proof of Proposition \ref{lem_analyticity}, we show that  $P_{g,f_g}:C^{k,\alpha}_{\delta}\to C^{k-2,\alpha}_{\delta}$ is an isomorphism. Because of that, we can carry out all nessecary estimates exactly as in \cite[Section 4]{Kro20}. The only difference is that one needs to replace the unweighted H\"{o}lder spaces by weighted ones and to use the trivial inclusion $C^{k,\alpha}_{\delta}\subset C^{k,\alpha}_{0}$.
\end{proof}
\begin{lem}\label{lem_divergence_control}
Let $h\in C^{k,\alpha}_{\delta}(S^2T^*M)$ be a tensor field satisfying  $\mathrm{div}_{\hg}h=0$.
For a PE metric $g\in \mathcal{R}^{k,\alpha}_{\delta}(M,\hg)$, we consider the $g$-dependent vector field $X_g$ satisfying
$\overline{h}_g=h-\mathcal{L}_{X_g}g\in\mathrm{ker}(\mathrm{div}_g)$ given by \eqref{eq_decomposition}. Then,
\begin{align*}
\left\|\mathcal{L}_{X_g}g\right\|_{H^1}\leq C \left\|g-\widehat{g}\right\|_{C^{k,\alpha}_{\delta}}\left\|h\right\|_{H^1}.
\end{align*}
\end{lem}
\begin{proof}
An analogous statement was shown in \cite[Lemma~5.4.4]{Kro13} for the compact setting and the proof is the same here. 
With respect to the orthogonal decomposition
\begin{align*}
H^l(TM)=\left\{\nabla_g f\mid f\in H^{l+1}(M)\right\}\oplus 
\left\{X\in  H^{l}(TM) \mid \mathrm{div}_gX=0\right\}
\end{align*}
for $l\in \left\{0,1,2\right\}$,
the operator $P:X\mapsto -\mathrm{div}_g\mathcal{L}_{X}g$ decomposes as
\begin{align*}
P=2\Delta_g\oplus (\Delta_g+(n-1))
\end{align*}
and both operators on the right hand side are isomorphisms from $H^2$ to $L^2$. Thus,
\begin{align*}
\left\|\mathcal{L}_{X_g}g\right\|_{H^1}
&\leq C \left\|X_g\right\|_{H^2}
\leq C \left\|PX_g\right\|_{L^2}
\leq C\left\|\mathrm{div}_gh\right\|_{L^2} \\
&\leq C\left\|(\mathrm{div}_g-\mathrm{div}_{\hg})h\right\|_{L^2}
\leq C \left\|g-\widehat{g}\right\|_{C^1}\left\|h\right\|_{H^1}.
\end{align*}
 The estimate of the lemma follows from the trivial inclusion $C^{k,\alpha}_{\delta}\subset C^1$.
\end{proof}

Now we are ready to prove the main result of this subsection.
\begin{thm}\label{thm:local_maximum_entropy}
Let the PE manifold $(M,\hg)$ be linearly stable and integrable. Then it is a local maximum of $\mu_{\rm{AH},\hg}$.
\end{thm}

\begin{proof}
Recall that by Lemma~\ref{lem:slice}, any metric $\widetilde{g}$ close to $\widehat{g}$ is isometric to a metric $g\in\mathcal{S}_{\hg}$ close to $\hg$. Therefore it suffices to prove
\begin{align*}
\mu_{\rm{AH},\hg}(g)\leq \mu_{\rm{AH},\hg}(\widehat{g})\qquad \text{for all }g\in\mathcal{S}_{\widehat{g}}\cap\mathcal{U}
\end{align*}
for a sufficiently small $C^{k,\alpha}_{\delta}$-neighbourhood $\mathcal{U}$ of $\hg$.
In order to prove this, let $\mathcal{E}$ be as in Definition~\ref{defn_lin_stab_int} and let $N$ be the orthogonal complement of $\ker_{L^2}(\Delta_E)$ in 
\begin{align*}
T_{\hg}S_{\hg}=\left\{h\in C^{k,\alpha}_{\delta}(S^2T^*M)\mid \mathrm{div}_{\hg}h=0\right\}.
\end{align*}
Since $\widehat{g}$ is integrable, $\mathcal{E}$ is a manifold and
\begin{align*}
T_{\hg}S_{\hg}=T_{\hg}\mathcal{E}\oplus N
\end{align*}
By the implicit function theorem applied to the map
\begin{align*}
\Psi: \mathcal{E}\times N\to T_{\hg}S_{\hg},\qquad (g,h)\mapsto g+h,
\end{align*}
there exists
a $C^{k,\alpha}_{\delta}$-neighbourhood $\mathcal{U}$ such that any $g\in\mathcal{S}_{\widehat{g}}\cap\mathcal{U}$ can be uniquely written as $g=\widetilde{g}+h$ with $\widetilde{g}\in \mathcal{E}$ and $h\in N$.
Let $h_{\widetilde{g}}$ be as in Lemma \ref{lem_divergence_control}.  Taylor expansion yields
\begin{align*}
\mu_{\rm{AH},\hg}(g)
&= \mu_{\rm{AH},\hg}(\widetilde{g})+\frac{d}{dt}\mu_{\rm{AH},\hg}(\widetilde{g}+th)\big|_{t=0}
+\frac{1}{2}\frac{d^2}{dt^2}\mu_{\rm{AH},\hg}(\widetilde{g}+th)\big|_{t=0}\\
&\qquad +\frac{1}{2}\int_0^1 (1-t)^2\frac{d^3}{dt^3}\mu_{\rm{AH},\hg}(\widetilde{g}+th)\\
&=\mu_{\rm{AH},\hg}(\hg)
-\frac{1}{4}\int_M \langle \Delta_{E,\widetilde{g}}h_{\widetilde{g}},h_{\widetilde{g}}\rangle_{\widetilde{g}}\dv_{\widetilde{g}}\\
&\qquad+\frac{1}{2}\int_0^1 (1-t)^2\frac{d^3}{dt^3}\mu_{\rm{AH},\hg}(\widetilde{g}+th).
\end{align*}
For the second equality, we used the fact that $\mu_{AH}$ is constant on the manifold $\mathcal{E}$ of its critical points. Let us now look more carefully at the term coming from the second variation.
By \cite[pp.\ 28--29]{Lic61}, see also \cite[Lemma 2.4.5]{Kro13}, we know that  $\Delta_{E,\widetilde{g}}$ preserves the splitting \eqref{eq_decomposition} according to which we have $h=h_{\widetilde{g}}+\mathcal{L}_{X_{\widetilde{g}}}\widetilde{g}$. Therefore,
\begin{align*}
\int_M \langle \Delta_{E,\widetilde{g}}h_{\widetilde{g}},h_{\widetilde{g}}\rangle_{\widetilde{g}}\dv_{\widetilde{g}}=
\int_M \langle \Delta_{E,\widetilde{g}}h,h\rangle_{\widetilde{g}}\dv_{\widetilde{g}}
-\int_M \langle \Delta_{E,\widetilde{g}}\mathcal{L}_{X_{\widetilde{g}}}\widetilde{g},\mathcal{L}_{X_{\widetilde{g}}}\widetilde{g}\rangle_{\widetilde{g}}\dv_{\widetilde{g}}.
\end{align*}
By Lemma \ref{lem_divergence_control},
\begin{align*}
\int_M \langle \Delta_{E,\widetilde{g}}\mathcal{L}_{X_{\widetilde{g}}}\widetilde{g},\mathcal{L}_{X_{\widetilde{g}}}\widetilde{g}\rangle_{\widetilde{g}}\dv_{\widetilde{g}}\leq C\left\|\mathcal{L}_{X_{\widetilde{g}}}\widetilde{g}
\right\|_{H^1}\leq C\left\|g-\widehat{g}\right\|_{C^{k,\alpha}_{\delta}}\left\|h\right\|_{H^1}
\end{align*}
and a Taylor expansion argument (see for example \cite[p.\ 74]{Kro13}) implies
\begin{align*}
\int_M \langle \Delta_{E,\widetilde{g}}h,h\rangle_{\widetilde{g}}\dv_{\widetilde{g}} \geq\int_M \langle \Delta_{E,\hg}h,h\rangle_{\hg}\dv_{\hg}- C\left\|g-\widehat{g}\right\|_{C^{k,\alpha}_{\delta}}\left\|h\right\|_{H^1}.
 \end{align*}
By \cite[Proposition D]{Lee06} we have that $\Delta_{E,\hg} \geq \frac{1}{4}(n-1)^2$ on $N$. Therefore, with a suitable choice of $\epsilon>0$, we get
\begin{align*}
\int_M \langle \Delta_{E,\hg}h,h\rangle_{\hg}\dv_{\hg}&=(1-\epsilon)\int_M \langle \Delta_{E,\hg}h,h\rangle_{\hg}\dv_{\hg}+\epsilon\int_M \langle \Delta_{E,\hg}h,h\rangle_{\hg}\dv_{\hg}\\
&\geq\frac{1-\epsilon}{4}(n-1)^2\left\|h\right\|_{L^2}^2+\epsilon\left\|\nabla h\right\|_{L^2}^2-\epsilon\left\|R_{\widehat{g}}\right\|_{L^{\infty}}\left\|h\right\|_{L^2}^2\\
&=\epsilon\left\|\nabla h\right\|_{L^2}^2
+\left[\frac{1-\epsilon}{4}(n-1)^2-\epsilon\left\|R_{\widehat{g}}\right\|_{L^{\infty}}\right]\left\|h\right\|_{L^2}^2\\
&\geq C\left\|h\right\|_{H^1}^2
\end{align*}
for some constant $C>0$. Combining these estimates, we have thus shown
\begin{align*}
\int_M \langle \Delta_{E,\widetilde{g}}h_{\widetilde{g}},h_{\widetilde{g}}\rangle_{\widetilde{g}}\dv_{\widetilde{g}}\geq C(1-\left\|g-\widehat{g}\right\|_{C^{k,\alpha}_{\delta}}) \left\|h\right\|_{H^1}^2
\end{align*}
which implies, together with Lemma \ref{lem_third_variation},
\begin{align*}
\mu_{\rm{AH},\hg}(g)
&= 
\mu_{\rm{AH},\hg}(\hg)
-\frac{1}{4}\int_M \langle \Delta_{E,\widetilde{g}}h_{\widetilde{g}},h_{\widetilde{g}}\rangle_{\widetilde{g}}\dv_{\widetilde{g}}\\
&\qquad+\frac{1}{2}\int_0^1 (1-t)^2\frac{d^3}{dt^3}\mu_{AH}(\widetilde{g}+th) \, dt\\
&\leq 
\mu_{AH}(\hg)-C(1-\left\|g-\widehat{g}\right\|_{C^{k,\alpha}_{\delta}}) \left\|h\right\|_{H^1}^2+\left\|h\right\|_{H^k}\left\|h\right\|_{H^1}^2.
\end{align*}
Thus we get the desired statement, provided that the $C^{k,\alpha}_{\delta}$-neighbourhood $\mathcal{U}$ is small enough.
\end{proof}

\subsection{Positivity of mass}\label{subsec:local_positive_mass}

As in Subsection \ref{subsec:conformal_PMT}, we set
\begin{align*}
\mathcal{C}=\left\{g\in \mathcal{R}^{k,\alpha}_{\delta}(M,\hg)\mid \scal_g=-n(n-1)\right\}.
\end{align*}
Observe that if $g\in \mathcal{C}$, the Euler--Lagrange equation \eqref{eq:EL_equation_entropy} implies that $f_g=0$, so that we have the identity $\mu_{\rm{AH},\hg}(g)=-\m_{\rm{VR},\hg}(g)$. 

\begin{prop}\label{prop_conformal_second_variation}
Let $g\in \mathcal{C}$. Then $g$ is a critical point of $\mu_{\rm{AH},\hg}$ with respect to conformal variations. Moreover, the second variation in conformal directions is given by
\begin{align*}
D_g^2\mu_{\rm{AH},\hg}(v g,v g)=-\int_M (Pv) v\dv,
\end{align*}
where the operator $P$ is given by
\begin{align*}
(n-1)(\Delta+n)\left(1-\frac{1}{2}\Delta(\Delta+(n-1))^{-1}\right).
\end{align*}
\end{prop}

\begin{proof}
Since $\scal_g=-n(n-1)$ we have $f_g= 0$ so the gradient of $\mu_{\rm{AH},\hg}$ at $g$ is trace-free. This proves the first assertion. For the second variation formula, we get as in the proof of Theorem \ref{thm:second_variation_entropy} that
\begin{align*}
&\frac{d^2}{dt^2}\mu_{\rm{AH},\hg}(g+th)\big|_{t=0}\\
&\qquad=-\int_M\langle \frac{1}{2}\left(\Delta_Lh+\mathcal{L}_{\mathrm{div}h}g-\nabla^2\trace h\right)+\nabla^2f'+(n-1)h,h\rangle  e^{-f}\dv.
\end{align*}
Continuing the computation with $h=v g$ yields
\begin{align*}
&\frac{d^2}{dt^2}\mu_{\rm{AH},\hg}((1+tv) g)\big|_{t=0}\\
&\qquad=-\frac{1}{2}\int_M\langle (\Delta v) g+(2-n)\nabla^2v+2\nabla^2f'+2(n-1)v g,v g\rangle \dv\\
&\qquad=-\frac{1}{2}\int_M (n\Delta v+(n-2)\Delta v-2\Delta f'+2n(n-1)v)v\dv\\
&\qquad=-\int_M \left((n-1)(\Delta v+nv)v-(\Delta f') v\right)\dv.
\end{align*}
To compute $f'$, we differentiate \eqref{eq:EL_equation_entropy}.
Since $f_g= 0$ is constant and $h=v g$, this yields
\begin{align*}
0&=2\Delta f'-\scal'+2(n-1)f'\\
&=2(\Delta+(n-1))(f')-\left(\Delta \trace h+\mathrm{div}(\mathrm{div} h)-\langle\ric,h\rangle\right)\\
&=2(\Delta+n-1)(f')-(n-1)(\Delta+n) v.
\end{align*}
Inserting this equation in the above yields the desired formula.
\end{proof}

\begin{rem}\label{rem:estimate_conf_second_variation}
Observe that
\begin{align*}
-\int_M (Pv) v\dv\leq -\frac{n-1}{2}\int_M (\Delta v+nv)v\dv
=-\frac{n-1}{2}\int_M (|\nabla v|^2+n v^2)\dv.
\end{align*}
\end{rem}

\begin{thm}\label{thm_local_positive_mass}
Let $(M,\hg)$ be a complete PE manifold. Then the following are equivalent.
\begin{itemize}
\item[(i)] $\hg$ is a local maximiser of $\mu_{\rm{AH},\hg}$.
\item[(ii)] For all metrics $g$ sufficiently close to $\hg$ with $\scal_g+n(n-1)$ nonnegative and integrable, we have $m_{\rm{VR},\hg}(g)\geq 0$. Moreover, equality holds if and only if $g$ is PE.
\item[(iii)] $\hg$ is a local minimizer of $m_{\rm{VR},\hg}$ on $\mathcal{C}$.
\item[(iv)] $\hg$ is a local maximiser of $\mu_{\rm{AH},\hg}$ on $\mathcal{C}$.
\end{itemize} 
\end{thm}
\begin{proof}
For proving $(i)\Rightarrow (ii)$, let $g$ be as in $(ii)$ and let $\omega_g$ be the minimizing function in the definition of $\mu_{\rm{AH},\hg}(g)$ through the functional $\widetilde{\mathcal{W}}_{\rm{AH},\hg}(g,\omega)$. Then, because $\scal_g+n(n-1)\geq 0$ and $\omega_g=e^{-f_g/2}\geq0$, we have
\begin{align*}
0=\mu_{\rm{AH},\hg}(\hg)&\geq \mu_{\rm{AH},\hg}(g)\\&=
\widetilde{\mathcal{W}}_{\rm{AH},\hg}(g,\omega_g)\\
&=\int_M \left(4|\nabla \omega_g|^2+(\scal_g +n(n-1))\omega_g^2\right) \dv\\
&\qquad +2(n-1)\int_M \left((\log(\omega_g^2)-1)\omega_g^2+1\right)\dv-\m_{\rm{VR},\widehat{g}}(g)\\
&\geq -\m_{\rm{VR},\hg}(g).
\end{align*}
Thus we have $\m_{\rm{VR},\hg}(g)\geq 0$. Moreoever, we get that $\m_{\rm{VR},\hg}(g)=0$ implies $\mu_{\rm{AH},\hg}(g)=0$. Therefore, $g$ is another local maximum of $\mu_{\rm{AH},\hg}$. In particular, it is a critical point and $\ric_g=-(n-1)g$ follows from Corollary \ref{cor:Entropy_critical_points}.

The implication $(ii)\Rightarrow (iii)$ is trivial. The implication $(iii)\Rightarrow (iv)$ follows immediately from the fact that $\mu_{\rm{AH},\hg}(g)=-\m_{\rm{VR},\hg}(g)$ for $g\in \mathcal{C}$.

For proving $(iv)\Rightarrow (i)$, we show that every metric $\olg\in \mathcal{C}$ is a local maximum of $\mu_{\rm{AH},\hg}$ in its conformal class. In fact, by Taylor expansion along the curve $\olg_t=\olg+tv\olg$, $t\in [0,1]$ and using Proposition \ref{prop_conformal_second_variation}, Remark \ref{rem:estimate_conf_second_variation} and Lemma \ref{lem_third_variation}, we obtain
\begin{align*}
\mu_{\rm{AH},\hg}((1+v)\olg)-\mu_{\rm{AH},\hg}(\olg)&=
\frac{d}{dt}\mu_{\rm{AH},\hg}(\olg_t)\big|_{t=0}+\frac{1}{2}\frac{d^2}{dt^2}\mu_{\rm{AH},\hg}(\olg_t)\big|_{t=0}\\
&\qquad+\frac{1}{2}\int_0^1(1-t)^2\frac{d^3}{dt^3}\mu_{\rm{AH},\hg}(\olg_t)\\
&\leq -\frac{n-1}{4}\left\|v\right\|_{H^1}^2+C\left\|v\right\|_{C^{k,\alpha}_{\tau}}\left\|v\right\|_{H^1}^2\\
&\leq -\left(\frac{n-1}{4}-C \left\|v\right\|_{C^{k,\alpha}_{\tau}}\right)\left\|v\right\|_{H^1}^2,
\end{align*}
 where the right hand side is nonpositive, provided that $v$ is sufficiently small. Now, let $g$ be any metric sufficiently close to $\hg$.
By Proposition \ref{prop_conformal_decomposition}, there exists a unique metric $\olg\in\mathcal{C}\cap [g]$ close to $g$. By the assumption in (ii) and the above computation,
\begin{align*}
\mu_{\rm{AH},\hg}(g)\leq \mu_{\rm{AH},\hg}(\olg)\leq \mu_{\rm{AH},\hg}(\hg),
\end{align*} 
which proves (i).
\end{proof}
\begin{rem}
The above proof shows that the equivalent assertions in Theorem \ref{thm_local_positive_mass} imply the bound
\begin{align*}
\m_{\rm{VR},\widehat{g}}(g)\geq \inf_{\omega-1\in C^{\infty}_c(M)}\int_M \left(4|\nabla \omega|^2+(\scal_g +n(n-1))\omega^2+F(\omega)
\right)\dv_g
\end{align*}
for all $g$ close to $\hg$ with $\scal_g +n(n-1)$ nonnegative and integrable. Here, $F$ denotes the nonnegative function $F(x)=2(n-1)\left((\log(x^2)-1)x^2+1\right)$ which vanishes exactly for $x=\pm 1$.
\end{rem}
The equivalent assertions in Theorem \ref{thm_local_positive_mass} imply scalar curvature rigid in the following sense.
\begin{cor}\label{cor:scr}
Let $(M,\hg)$ be a complete PE manifold which satisfies the equivalent assertions of Theorem \ref{thm_local_positive_mass}.
Then, any metric $g$ on $M$ sufficiently close to $\hg$ with $\scal_{g}\geq \scal_{\hg}$, for which there exists a compact set $K\subset M$ with
\begin{align}\label{eq:cpsupp_deformation}
g-\hg|_{M\setminus K}= 0,\qquad \volume(K,g)=\volume(K,\hg),
\end{align}
is isometric to $\hg$.
\end{cor}
\begin{proof}
By \eqref{eq:cpsupp_deformation}, $\m_{\rm{VR},\hg}(g)=0$. By Theorem \ref{thm_local_positive_mass} (ii), $(M,g)$ is a complete PE manifold. Since $g$ and $\hg$ are both Einstein and agree outside a compact set, they have to be isometric, see \cite{deturckkazdan1981}.
\end{proof}

\begin{exmp}
	Hyperbolic space is well known to satisfy the spectral inequality $\Delta_E\geq \frac{1}{4}(n-1)^2$, so it is linearly stable and integrable. By Theorem \ref{thm:local_maximum_entropy}, it is a local maximum of the entropy and by Theorem \ref{thm_local_positive_mass}, it is a local minimizer of the volume-renormalized mass.
\end{exmp}

%
\section*{Declarations}
\subsection*{Data availability.} No data associate for the submission.

\subsection*{Conflict of interest.} On behalf of all authors, the corresponding author states that there is no conflict of
interest.


\bibliographystyle{abbrvurl}
\bibliography{AHmass}

\end{document}